\documentclass[a4paper]{article}
\usepackage{amsthm,amsfonts,amsmath,amssymb,units}
\usepackage[abbrev,nobysame]{amsrefs}
\usepackage[cp1251]{inputenc}
\usepackage[english]{babel}
\usepackage[final]{graphicx}
\usepackage{setspace}
\usepackage[12pt]{extsizes}
\begin{document}
\newtheorem{teorema}{Theorem}
\newtheorem{lemma}{Lemma}
\newtheorem{utv}{Proposition}
\newtheorem{svoistvo}{Property}
\newtheorem{sled}{Corollary}
\newtheorem{con}{Conjecture}
\newtheorem{zam}{Remark}
\newtheorem{quest}{Question}
\newtheorem{problem}{Problem}
\newtheorem{example}{Example}

\author{A. A. Taranenko\thanks{Sobolev Institute of Mathematics, Novosibirsk, Russia; \texttt{taa@math.nsc.ru}}}
\title{Perfect colorings of hypergraphs}
\date{October 24, 2024}

\maketitle

\begin{abstract} 
Perfect colorings (equitable partitions) of graphs are extensively studied, while the same concept for hypergraphs attracts much less attention. The aim of this paper is to develop basic notions and properties of perfect colorings for hypergraphs.  Firstly, we introduce a multidimensional matrix equation for perfect colorings of hypergraphs and compare this definition with a standard approach  based on the incidence graph.   Next, we show that  the eigenvalues of the parameter matrix of a perfect coloring are eigenvalues of the multidimensional adjacency matrix of a hypergraph. We consider coverings of hypergraphs as a special case of perfect colorings and prove a theorem on the existence of a common covering of two hypergraphs. As an example, we show that a $k$-transversal in a hypergraph corresponds to a perfect coloring and calculate its parameters.  At last, we find all perfect $2$-colorings of the Fano's plane hypergraph and compute some eigenvalues of this hypergraph.


\end{abstract}

\section{Introduction}

A coloring of vertices of a graph is called perfect if  a color of a vertex uniquely defines colors of adjacent vertices.

Perfect colorings of graphs have been studied for quite a long time and arose in the literature under different names. One of the most known of them is ``equitable partition'' that was introduced by Delsarte~\cite{dels.assch}. In book~\cite{CveDoobSac.gspectra} objects equivalent to perfect colorings are named as divisors of graphs. An algebraic definition of perfect colorings firstly appeared in book~\cite{godsil.algcomb}. At last, in~\cite{GodRoy.alggrath} one finds some information on graph coverings that can be considered as perfect colorings of a special type.

Graph perfect colorings  have plenty of applications in coding theory and the theory of combinatorial designs.  There are many results on enumeration or characterization of perfect colorings in certain families of graphs (see, e.g., survey papers~\cite{our.perfcolorHam, BorRifaZin.compregcodes}).

Meanwhile, the specifics of  perfect colorings in hypergraphs are hardly studied. As it was noted  in  ``Handbook of combinatorics''~\cite[Ch. 31.2]{GraGroLo.handbook}, a perfect coloring of a hypergraph is equivalent to a certain perfect coloring of its incidence graph.  In~\cite{PotAv.hypercolor}, it was shown that some combinatorial designs correspond to perfect colorings of hypergraphs. In particular, every transversal in a regular uniform hypergraph is a perfect $2$-coloring.

Reduction of perfect colorings of hypergraphs to colorings of graphs   implies immediately  some of their properties. On the other  hand, this approach uses redundant parameters of a coloring of hyperedges that can be deduced from a coloring of vertices.

The main aim of this paper is to propose and develop another perspective for perfect colorings of hypergraphs based on their multidimensional adjacency  matrices.

The paper is structured as follows. 

Section~\ref{prelim-section} is preliminary. It provides all definitions and auxiliary results on hypergraphs, their adjacency and incidence matrices, multidimensional matrices and their products, and eigenvalues of multidimensional matrices and hypergraphs. In this section we also briefly review the main results on perfect colorings in graphs that will be later generalized for hypergraphs. 

Section~\ref{main-section} is the main part of the paper.  Firstly, we consider in detail the definition of a hypergraph perfect coloring as a perfect coloring of its incidence graph. Although the concept of the hypergraph perfect coloring was introduced  earlier, it  is a  debatable  question how to  define their parameters  so that they still have the most nice properties of the graph case. Here we pay special attention to this question because  it was omitted in  previous works. Also we deduce some simple corollaries from  the approach based on the incidence graph that were not given before.  In particular, we get necessary and sufficient conditions for a pair of matrices to be parameters of a perfect coloring of some hypergraph  (Theorem~\ref{VWblock}) and establish the analogue of the Weisfeiler--Leman--Vizing theorem on the refinement of perfect coloring (Theorem~\ref{WLhyper}). 

In the second part of Section~\ref{main-section}, we propose  a new definition of a hypergraph perfect coloring in terms of its multidimensional adjacency matrix. In Theorem~\ref{adjparam} we show that for uniform hypergraphs this definition is equivalent to the previous one  and  introduce the multidimensional parameter matrix of a perfect coloring.  We also prove that the multidimensional parameter matrix can be symmetrized (Theorem~\ref{symmetrparam}) and connect it with the $2$-dimensional parameter matrices of the coloring (Corollary~\ref{parammultii2}). The major asset of the multidimensional parameter matrix is that  all  its eigenvalues are  also the eigenvalues of the adjacency matrix of a hypergraph (Theorem~\ref{adjcpecinclude}). We obtain it as a corollary  from  a   more  general statement for multidimensional matrices (Theorem~\ref{specinclude}). These results  allow one to find several eigenvalues of a multidimensional matrix (or hypergraph) on the basis of eigenvalues of certain smaller matrices (hypergraphs).

In Section~\ref{covering-setion}, we consider coverings of hypergraphs as a special case of perfect colorings. This allows us to easily get the following properties of coverings: characterization of hypergraphs that cover a given one (Theorem~\ref{coverinci}), preservation of parameters of a perfect coloring in coverings (Theorem~\ref{covercolorinclude}), and inclusion of the spectrum of hypergraphs into the spectrum of their coverings (Theorem~\ref{coverspecinclude}). The last result was  recently obtained in~\cite{SongFanWang.hyperspeccover} by another method.  At the end of the section we generalize  Leighton's theorem~\cite{leighton.comcover} on  common coverings from graphs to  hypergraphs (Theorem~\ref{commoncover} and Corollary~\ref{Leightonhyper}).  

At last, in Section~\ref{example-section} we consider several  examples of hypergraph perfect colorings. We find the parameter matrices of a hypergraph transversal as a perfect $2$-coloring and calculate their eigenvalues (Theorem~\ref{transspectrum}). Moreover, we show that multidimensional parameters may distinguish hypergraphs with no transversals, while the $2$-dimensional approach does not.  Finally, we compute the parameters of  perfect $2$-colorings of regular $3$-uniform hypergraphs and  find  all perfect $2$-colorings of the Fano's plane hypergraph  and some of its eigenvalues.

\section{Definitions and preliminaries} \label{prelim-section}

A hypergraph $\mathcal{G}$ is a pair of sets $X = X(\mathcal{G})$ and $E = E(\mathcal{G})$, $|X| = n$, $|E| = m$. $X$ is called a set of \textit{vertices} and $E$ is a set of \textit{hyperedges},  each $e \in E$ is some (nonempty) subset of $X$.  A vertex $x \in X$ is said to be \textit{incident} to a hyperedge $e $ if  $x \in e$.  Vertices $x$ and $y$ are \textit{adjacent} if there is $e \in E$ incident to both of them.    If $E$ is a multiset, then $\mathcal{G}$ is said to be a \textit{multihypergraph}. 

A hypergraph is  \textit{connected} if there is an interchanging sequence of incident vertices and hyperedges that connects a vertex $x$ with any other vertex $y$. Such a sequence is called a \textit{Berge path} from $x$ to $y$.  

A hypergraph $\mathcal{G}$ is said to be \textit{$d$-uniform}  if every hyperedge of $\mathcal{G}$ consists of $d$ vertices. Simple graphs are  exactly  $2$-uniform hypergraphs. A $d$-uniform hypergraph is \textit{complete} if its hyperedges are all $d$-element subsets of the vertex set.

A \textit{degree} of a vertex $x$ is the number of hyperedges containing $x$. A hypergraph $\mathcal{G}$ is \textit{$r$-regular} if every vertex of $\mathcal{G}$ has degree $r$. 

A hypergraph $\mathcal{G}$ is said to be \textit{$d$-partite} if there is a partition of its vertex set into $d$ disjoint subsets (\textit{parts}) such that each hyperedge of $\mathcal{G}$ contains no more than one vertex from each part. 

A \textit{$k$-transversal} in a hypergraph $\mathcal{G} = (X,E)$ is a set of vertices $Y \subseteq X$ such that every hyperedge of $\mathcal{G}$ contains exactly $k$ vertices from $Y$. A \textit{$k$-factor} in $\mathcal{G}$ is a set of hyperedges $U \subseteq E$ such that every vertex of $\mathcal{G}$ is incident to  exactly $k$ hyperedges from $U$. $1$-factors  are said to be \textit{perfect matchings} and $1$-transversals are known as \textit{transversals}. 

Transversals and factors are dual in the following sense. For every hypergraph $\mathcal{G}$ one can consider a dual hypergraph $\mathcal{G}^*$ obtained from $\mathcal{G}$ by interchanging sets of vertices and hyperedges and preserving the incidence relations between them.  Then every $k$-transversal in $\mathcal{G}$ is a $k$-factor in $\mathcal{G}^*$ and vice versa.

Every hypergraph $\mathcal{G} = (X,E)$ corresponds to the \textit{incidence graph} $G$ that is  a bipartite graph  with parts $X$ and $E$: $x \in X$ and $e \in E$ are adjacent in $G$ if and only if they are incident in $\mathcal{G}$.  Incidence graph is also known as the Levi graph or the bipartite representation of a hypergraph.

\subsection{Matrices and eigenvalues of hypergraphs}

Let $A = \{ a_{i,j}\} $, $i = 1, \ldots, n$, $j = 1, \ldots, m$ be a matrix of size $n \times m$. If $n = m$ then $A$ is a matrix of \textit{order} $n$.

We will say that a matrix $A$ of size $n \times m$ is a \textit{block matrix} with \textit{blocks} $B_{i,j}$ of sizes $n_i \times m_j$, $i = 1, \ldots, k$, $j = 1, \ldots, l$, $\sum\limits_{i=1}^k n_i = n$,  $\sum\limits_{i=1}^l m_j = m$, if  (after, possibly, appropriate permutations of rows and columns) the matrix $A$ can be presented in a form
$$
A = \left( 
\begin{array}{ccc}
B_{1,1} & \cdots & B_{1,l} \\
\vdots & \ddots & \vdots \\
B_{k,1} & \cdots & B_{k,l}
\end{array}
\right).
$$
For shortness, we write $A = \{B_{i,j} \}$.

Let $[n] = \{ 1, \dots, n\}$. A \textit{$d$-dimensional matrix $\mathbb{A}$ of order $n$} is an array of entries $(a_{\alpha})$, $a_{\alpha} \in \mathbb{R}$, indexed by tuples $\alpha \in [n]^d$, $\alpha = (\alpha_1, \ldots, \alpha_d)$. Vectors can be considered as $1$-dimensional matrices, and usual matrices have $2$ dimensions. 

Let $\mathbb{A}$ be a $d$-dimensional matrix of order $n$.
A \textit{hyperplane} $\Gamma$ of \textit{direction} $i$ in $\mathbb{A}$ is a $(d-1)$-dimensional submatrix obtained by fixing a value of index component $\alpha_i$ and letting the other $d-1$ components vary. 

A matrix $\mathbb{A}$ is said to be \textit{symmetric} if for every permutation $\sigma \in S_d$ and for every index $\alpha \in [n]^d$ we have $a_{\alpha} = a_{\sigma(\alpha)}$, where $\sigma(\alpha) = (\alpha_{\sigma(1)}, \ldots, \alpha_{\sigma(d)})$. In other words, the matrix $\mathbb{A}$ does not change under permutations of directions of hyperplanes. 

The \textit{$d$-dimensional identity matrix $\mathbb{I}$} is the matrix with entries $i_{\alpha} = 1$ if $\alpha_1 = \cdots = \alpha_d$ and $i_\alpha = 0$ otherwise.

Let $\mathcal{G} = (X,E)$ be hypergraph on $n$ vertices with $m$ hyperedges.
The \textit{incidence matrix}  $B$  of the hypergraph  $\mathcal{G}$ is a matrix of size $n \times m$ such that $b_{x,e} = 1$ if a vertex $x$ is incident to a hyperedge $e$ and  $b_{x,e} = 0$ otherwise.

If $\mathcal{G}$ is a uniform hypergraph, then we can  define its multidimensional adjacency matrix.
The \textit{adjacency matrix} $\mathbb{A}$ of a $d$-uniform hypergraph $\mathcal{G} = (X,E)$ on $n$ vertices is a $d$-dimensional matrix of order $n$ in which entries $a_\alpha$ with index $\alpha = (x_1, \ldots, x_d) \in E$ are equal to $(d-1)!^{-1}$  and all other entries of $\mathbb{A}$ are $0$. Such scale of entries of the matrix $\mathbb{A}$ is taken for the sake of compactness of future expressions.  By definition, the adjacency matrix $\mathbb{A}$ is  symmetric.

If $G$ is a graph, then an \textit{eigenvalue} of $G$ is an eigenvalue of its $2$-dimensional adjacency matrix.  For hypergraphs, there are several ways to define eigenvalues, the most popular approaches use the following:
\begin{enumerate}
\item maximization of some multilinear form generated by a hypergraph;
\item eigenvalues of the $2$-dimensional signless Laplacian matrix $BB^{T}$;
\item eigenvalues of the multidimensional adjacency matrix $\mathbb{A}$.
\end{enumerate}

In this paper we will utilize the last of them. For a brief survey on other definitions of hypergraph eigenvalues see, e.g., the introduction of~\cite{coop.specrandcomp}.
To define eigenvalues of multidimensional matrices, we use the following operation.

 Let  $\mathbb{A}$ be a $d$-dimensional matrix of order $n$  and  $\mathbb{B}$ be a $t$-dimensional matrix of the same order.   Define the \textit{product} $\mathbb{A} \circ \mathbb{B}$ to be the $((d -1) (t -1) +1)$-dimensional  matrix $\mathbb{C}$ of order $n$ with entries 
 $$c_{i, \beta^2, \ldots, \beta^{d}} = \sum\limits_{i_2= 1}^n \cdots \sum\limits_{i_d= 1}^n  a_{i, i_2, \ldots, i_{d}} \cdot b_{i_2, \beta^2}\cdots b_{i_{d}, \beta^{d}},$$
 where indices $\beta^{2}, \ldots, \beta^{d} \in [n]^{t-1}$, $i \in [n]$.

The following properties of the product $\circ$  can be derived directly from the definition or  found in~\cite{shao.tensprod}.

\begin{utv} \label{prodcomute}
 Let $\mathbb{A}$, $\mathbb{B}$, and $\mathbb{C}$ be multidimensional  matrices of appropriate sizes.
\begin{enumerate}
\item If dimensions of matrices $\mathbb{A}$ and $\mathbb{B}$ do not exceed $2$, then $\circ$ is a standard matrix (dot) product: $\mathbb{A} \circ \mathbb{B} = \mathbb{A} \mathbb{B}$.
\item The product of multidimensional matrices is associative:  $(\mathbb{A} \circ \mathbb{B}) \circ \mathbb{C} = \mathbb{A} \circ (\mathbb{B} \circ \mathbb{C}) $.
\item If $\lambda \in \mathbb{R}$ and $\mathbb{A}$ is a $d$-dimensional matrix, then $\mathbb{A} \circ (\lambda \mathbb{B}) = \lambda^{d-1} (\mathbb{A} \circ \mathbb{B}$).
\end{enumerate}
\end{utv}

Let $\mathbb{A}$ be a $d$-dimensional matrix of order $n$. We will say that $\lambda$ is an \textit{eigenvalue} of $\mathbb{A}$ if there is a vector $x = (x_1, \ldots, x_n)$ such that $A \circ x = \lambda \mathbb{I} \circ x $, where $\mathbb{I}$ is a $d$-dimensional identity matrix of order $n$ and $\mathbb{I} \circ x$ is a vector $(x_1^{d-1}, \ldots, x_n^{d-1})$. The  vector $x$ is called an \textit{eigenvector} corresponding to the eigenvalue $\lambda$,  a pair $(\lambda, x)$ is said to be an \textit{eigenpair} for $\mathbb{A}$.

If $\mathcal{G}$  is a $d$-uniform hypergraph with the $d$-dimensional adjacency $\mathbb{A}$, then the eigenvalues of the matrix  $\mathbb{A}$ are said to be the \textit{eigenvalues} of  $\mathcal{G}$,  the eigenvectors of $\mathbb{A}$ are the \textit{eigenvectors} of $\mathcal{G}$.

The set $V(\lambda) \subset \mathbb{C}^n$ of all eigenvectors for an eigenvalue $\lambda$ is a complex algebraic variety, and the \textit{geometric multiplicity} of the eigenvalue $\lambda$ is the dimension of  $V(\lambda)$. 

The present definition of eigenvalues and eigenvectors for  tensors was proposed in 2005 by Lim~\cite{lim.eigentensor} and was studied for symmetric matrices by Qi~\cite{qi.eigentensor}. They used the theory of determinants of multidimensional matrices and resultants of multilinear systems developed in book~\cite{GeKaZe.multdet} by Gelfand, Kapranov, and Zelevinsky. In particular, they proved that all eigenvalues of a multidimensional matrix $\mathbb{A}$ are roots of its \textit{characteristic polynomial} $\varphi_{\mathbb{A}}$.

\begin{teorema}[\cite{qi.eigentensor}]
For a $d$-dimensional matrix $\mathbb{A}$ of order $n$, there is a characteristic polynomial $\varphi_{\mathbb{A}}$ of degree $n(d-1)^{n-1}$ such that a number $\lambda \in \mathbb{C}$ is an eigenvalue of $\mathbb{A}$ if and only if $\lambda$ is a root of $\varphi_{\mathbb{A}}$.
\end{teorema}

The \textit{algebraic multiplicity} of an eigenvalue $\lambda$ is said be its multiplicity as a root of the characteristic polynomial. It is well known that for $2$-dimensional matrices geometric and algebraic multiplicities coincide. Unfortunately, in the multidimensional case it is not true. For  more information on relations between these multiplicities see~\cite{HuYe.multeigen}.

The study of eigenvalues of multidimensional matrices and hypergraphs grows now extensively, so we mention here only several papers. In  \cite{CoopDut.hyperspec} Cooper and Dutle  proved some properties of hypergraph spectra and calculate eigenvalues of certain classes of hypergraphs.   In \cite{shao.tensprod} it was introduced several products of hypergraphs and multidimensional matrices and then considered the eigenvalues of the result.  Algebraic properties of eigenvalues and determinats of multidimensional matrices with respect to various matrix operations were studied in \cite{ShaoShanZhang.detoftens}. At last, for some survey on the spectral theory of nonnegative matrices see~\cite{ChaQiZhang.survspecthyper}.

\subsection{Perfect colorings and coverings of graphs}

Let us review some basic definitions and facts  on perfect colorings of graphs that we aim to generalize for hypergraphs. Most of them can be found in~\cite{my.perfstrct}.

A \textit{coloring} of a graph $G = (X,E)$ in $k$ colors is a surjective function $f: X \rightarrow \{1, \ldots, k \}$. To every $k$-coloring $f$, we can put into a correspondence a rectangular \textit{color matrix $P$} of size $|X| \times k$ with entries $p_{x,i} = 1$ if $f(x) = i$ and $p_{x,i} = 0$ otherwise. Note that each row of $P$ contains exactly one nonzero entry. A coloring of a graph parts its vertex set $X$ into disjoint \textit{color classes} $\{C_1, \ldots, C_k \}$, where $C_i  = \{ x: f(x) = i \}$.  In what follows, we usually define a coloring with the help of the color matrix $P$, but in some cases we also refer to the coloring function $f$ and the color classes $\{C_1, \ldots, C_k \}$. 

If $G$ is a bipartite graph and $f$ is a coloring of $G$, we will say that $f$ is a \textit{bipartite} coloring if there are no color classes of $f$ that intersect both parts of $G$.

A $k$-coloring $f$ of a graph $G$ is called \textit{perfect} if there exist integer $s_{i,j}$, $i,j \in \{ 1, \ldots, k\}$,  such that every vertex of color $i$ in $f$ is adjacent to exactly $s_{i,j}$ vertices of color $j$. The matrix $S = (s_{i,j})$ of order $k$ is called the \textit{parameter matrix} of the perfect coloring. 

It is not hard to see that if $A$ is the adjacency matrix of a graph $G$, then $P$ is a perfect coloring with the parameter matrix $S$ if and only if $AP = PS$. Moreover, if $y$ is an eigenvector of the parameter matrix $S$ with an eigenvalue $\lambda$, then $Py$ is the eigenvector of $A$ with the same eigenvalue $\lambda$. 

Given a coloring $P$, let us denote by $N$ the matrix $P^T P$. It can be checked that $N$ is a diagonal  matrix with entries $n_{i,i}$ equal to the number of vertices of color $i$ and that the matrix $NS = P^TPS$ is symmetric.

We will say that a coloring $f$ of a graph $G$ with color classes $\{ C_1, \ldots, C_k\}$ is a \textit{refinement} of a coloring $g$ with color classes $\{ D_1, \ldots, D_l\}$ if each color class $D_i$ is a union of some classes $C_j$. For a given graph $G$, its colorings (and perfect colorings) compose a partially ordered set with respect to the refinement.

For every coloring of a graph, there is the unique coarsest refinement that will be a perfect coloring. The method of finding such refinements is known as Weisfeiler--Leman algorithm~\cite{WeiLeh.algo}.

\begin{teorema}[\cite{WeiLeh.algo}] \label{WLgraph}
Given a graph $G$, there is a perfect coloring $f$ such that every other perfect coloring of $G$ is a refinement of $f$. Moreover, for every coloring $g$ of $G$ there is a refinement $h$ of $g$ such that $h$ is a perfect coloring, and any other perfect coloring that refines $g$ is a refinement of $h$. 
\end{teorema}

A graph $G = (X,E)$ \textit{covers} a graph $H = (Y,U)$ if there exists a surjective function $\varphi : X \rightarrow Y$ such that for each $x \in X$ the equality $\{\varphi(x')|(x, x') \in E\} = \{y|(y, \varphi(x)) \in U \}$ holds. 
A covering $\varphi$ of a graph $H$ by $G$  is a \textit{$k$-covering}  if for every $y \in Y$ there are exactly $k$ vertices $x \in X$ such that $\varphi(x) = y$. It is well known that every covering is a $k$-covering for some $k \in \mathbb{N}$. 

It is not hard to see that a graph $G$ covers a graph $H$ if and only if there is a perfect coloring of $G$ with the parameter matrix equal to the adjacency matrix of $H$. It implies, for example, the following simple properties of coverings:
\begin{itemize}
\item If a graph $G$ covers a graph $H$ and a graph $H$ covers a graph $F$, then $G$ covers $F$.
\item If a graph $G$ covers a graph $H$ and $\lambda$ is an eigenvalue of $H$, then $\lambda$ is an eigenvalue of $G$.
\end{itemize}

In~\cite{leighton.comcover, AngGard.comcover}, it was proved that if graphs $H_1$ and $H_2$ have perfect colorings with the same parameter matrices, then there exists a graph $G$ that covers both $H_1$ and $H_2$. 

 For more information of graph coverings and their applications see~\cite{GodRoy.alggrath}.

\section{Perfect colorings of hypergraphs} \label{main-section}

There are two equivalent ways to define perfect colorings in uniform hypergraphs. The first of them uses a perfect coloring of the incidence graph, and the second one is based on an equation for  the multidimensional adjacency matrix of a hypergraph.  The first approach  is more general and can be applied to non-uniform hypergraphs.  Although  it appeared  in~\cite{GraGroLo.handbook} and~\cite{PotAv.hypercolor},  the parameters of such colorings were not discussed before. The approach based on multidimensional matrices is completely new.  In this section we also study   the interrelations between these approaches.

\subsection{Colorings of the incidence graph}

Let $\mathcal{G} (X,E)$ be a hypergraph.  We will say that a surjective function $f: X \rightarrow \{1, \ldots, k \}$ is a \textit{coloring}  of $\mathcal{G}$ into $k$ colors (or \textit{$k$-coloring}). In other words, $f$ defines a partition of the set $X$ in $k$ \textit{color classes}.  Given a coloring $f$ and a hyperedge $e$, let the \textit{color range} $f(e)$  be the multiset of colors of all incident vertices: $f(e) = \{ f(x) | x \in e \}  $.

Let $G$ be the incidence graph of the hypergraph $\mathcal{G}$.
To  each coloring  $f$ of $\mathcal{G}$, we associate an \textit{induced}  coloring $g$ of $G$ such that for every $x\in X$ we have $g(x) = f(x)$ and for every $e \in E$ we define $g(e) = f(e)$. Here   the color ranges $f(e)$ of hyperedges   are considered as  colors of the coloring $g$.  Note that every induced coloring of $G$ is bipartite, but not every bipartite coloring of $G$ is induced by some coloring of $\mathcal{G}$.

A coloring $f$ of a hypergraph $\mathcal{G}$ is said to be \textit{perfect} if each two vertices $x$ and $y$ of the same color have the same set of color ranges of incident hyperedges. Directly from the definitions, it follows  that $f$ is a perfect coloring of $\mathcal{G}$ if and only if  the induced bipartite coloring $g$ of $G$ is perfect.

We will say that  a rectangular $(0,1)$-matrix $P$ is a \textit{color matrix} if every row of $P$ contains exactly one $1$.
To every $k$-coloring $f$ of a hypergraph $\mathcal{G}(X,E)$ on $n$ vertices, we associate the color matrix  $P$ of size $n \times k$ such that $p_{x,i} = 1$ if $f(x) = i$ and $p_{x,i} = 0$ otherwise.

Assume that a hypergraph $\mathcal{G} (X,E)$ has $n$ vertices, $m$ hyperedges, and the incidence matrix $B$. Let $f$ be a perfect $k$-coloring of  $\mathcal{G}$ such that the number of different color ranges $\gamma$ of hyperedges is equal to $l$.   Then  the induced coloring $g$ of the incidence graph $G$ gives us  the following matrix equation:
$$
\left( \begin{array}{cc} 0 & B \\ B^T & 0 \end{array} \right)
\left( \begin{array}{cc} 0 & P \\ R & 0 \end{array} \right) = 
\left( \begin{array}{cc} 0 & P \\ R & 0 \end{array} \right) 
\left( \begin{array}{cc} 0 & W \\ V & 0 \end{array} \right)
$$
or, equivalently,
\begin{equation} \label{percoldef} 
B R = P V; ~~~ B^T P = R W, 
\end{equation}
where
\begin{itemize}
\item $P$ and $R$ are vertex and  hyperedge color matrices of sizes $n \times k$ and $m \times l$, respectively; the coloring $R$ is \textit{induced} by the coloring $P$.
\item $V$ and $W$ are  matrices of sizes  $k \times l$ and $l \times k$, respectively. An entry $v_{i,\gamma}$ of $V$ is equal to the number of hyperedges of color range $\gamma$ in $\mathcal{G}$ incident to a vertex of color $i$, entry $w_{\gamma,i}$ of $W$ is equal to the  number of vertices of color $i$ in $\mathcal{G}$ contained in  a hyperedge of color  range $\gamma$. 
\end{itemize}

We will say that the pair $(V,W)$ is the \textit{incidence parameters} of the perfect coloring $f$: the matrix $V$ is said to be the \textit{XE-parameter matrix} (describing incidence of vertices $x$ to hyperedges $e$), and $W$ is the \textit{EX-parameter matrix}   (describing incidence of hyperedges $e$ to vertices $x$).

\textbf{Example 1.}  Let $\mathcal{G}$ be a $3$-uniform hypergraph  with the vertex set $X = \{ x_1, x_2, x_3, x_4, x_5, x_6\}$ and the hyperedge set $E = \{ e_1, e_2, e_3, e_4\}$, where $e_1 = \{ x_1, x_2, x_3\}$, $e_2 = \{ x_1, x_4, x_5\}$,  $e_3 = \{ x_2, x_4, x_6\}$,  $e_4 = \{ x_3, x_5, x_6\}$.  

Let $f$ be a coloring of $X$ into two colors such that $f(x_1) = f(x_2) = f(x_3) = 1$ and  $f(x_4) = f(x_5) = f(x_6) = 2$. Then the color ranges of hyperedges are $f(e_1) = \{ 1, 1, 1\}$ and $f(e_2) = f(e_3)  = f(e_4)= \{ 1, 2, 2\}$. Note that each vertex of color $1$ is incident to one hyperedge of color range $\{ 1, 1, 1\}$ and one hyperedge of color range $\{ 1, 2, 2\}$, and each vertex of color $2$ is incident to two hyperedges of color range $\{ 1, 2, 2\}$. Therefore, $f$ is a perfect coloring of $\mathcal{G}$.

The incidence graph $G$ of $\mathcal{G}$ with the induced perfect coloring $g$ is given at Figure~1.

\begin{center}
\includegraphics[width=0.2\linewidth]{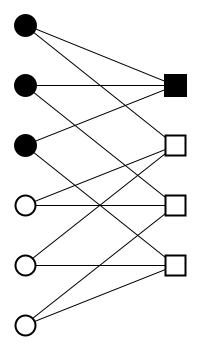}  

Figure 1: The perfect coloring $g$ of the  incidence graph $G$ for the hypergraph $\mathcal{G}$
\end{center}

The incidence matrix of the hypergraph $G$ is
$$B = \left( 
\begin{array}{cccc}
1 & 1 & 0 & 0 \\
1 & 0 & 1 & 0 \\
1 & 0 & 0 & 1 \\
0 & 1 & 1 & 0 \\
0 & 1 & 0 & 1 \\
0 & 0 & 1 & 1 \\
\end{array}
\right).$$

By the definition, the vertex color matrix $P$ and the hyperedge color matrix $R$  for  the perfect  coloring $f$ are
$$ P = 
\left(  \begin{array}{cc}
1 & 0 \\ 1 & 0 \\ 1 & 0 \\
0 & 1 \\ 0 & 1 \\ 0 & 1
\end{array} \right); ~~~
R = \left( \begin{array}{cc}
1 & 0 \\ 0 & 1  \\ 0 & 1 \\ 0 & 1 
\end{array} \right),$$
and the incidence parameters $(V, W)$ are
$$V = \left( 
\begin{array}{cc}
1 & 1 \\ 0 & 2 
\end{array}
\right); ~~~ 
W = 
\left( \begin{array}{cc}
3 & 0 \\ 1 & 2 
\end{array} \right).$$
It can be checked directly that $BR = PV$ and $B^TP  = RW$. 
\bigskip

Let us derive several simple observations from the given definition of a perfect coloring of hypergraphs.

First of all, if vertices $x$ and $y$ of a hypergraph $\mathcal{G}$ have different degrees, then $f(x) \neq f(y)$  for every  perfect coloring $f$ of $\mathcal{G}$ because the same is true for perfect colorings of graphs.  The sum of entries in the $i$-th row of the XE-parameter matrix $V$ is equal to the degree of vertices of color $i$, and the sum of entries in the $\gamma$-th row of the EX-parameter matrix $W$ is equal to the size of a hyperedge of the color range $\gamma$.  

Next, if $P$ is a coloring of vertices and $R$ is a coloring of hyperedges, then  $N = P^T P$ and $M = R^T R$ are diagonal matrices. Entries $n_{i,i}$ of  $N$ are equal to the numbers of vertices colored by color $i$ in $P$, and entries  $m_{j,j}$ of  $M$ are the numbers of hyperedges colored by color $j$ in $R$.

We also prove the following results on the structure of the incidence matrix $B$ and a relation between the incidence parameters $(V,W)$ and matrices $N$ and $M$.

\begin{teorema} \label{VWblock}
\begin{enumerate}
\item  Let $\mathcal{G}$ be a hypergraph with the incidence matrix $B$, $P$ be a coloring of $\mathcal{G}$, $R$ be an induced coloring of hyperedges, $N = P^T P$ and $M = R^T R$ be the diagonal matrices with entries $n_i$ and $m_j$ at the diagonal. If a coloring $P$ is perfect for $\mathcal{G}$ with incidence parameters $(V,W)$, then  $ N V =W^T M$ and $B$ is a  block matrix $\{ A_{i,j}\}$, where $A_{i,j}$ are $(0,1)$-matrices of sizes $n_i \times m_j$  with row sums $v_{i,j}$  and column sums $w_{j,i}$.
\item Let $V$ and $W$ be integer nonnegative matrices and $N$ and $M$ be integer diagonal matrices with $n_{i} > 0$ and $m_{j} >0$ such that $ N V =W^T M$. Then there exist a hypergraph $\mathcal{G}$ and its perfect coloring $P$ having incidence parameters $(V,W)$. 
\end{enumerate}
\end{teorema}

\begin{proof}
(1). Assume that $P$ is a perfect coloring of $\mathcal{G}$.  Recall that equations~(\ref{percoldef}) for perfect colorings give 
$$B R = P V; ~~~ B^T P = R W.$$
Multiplying  the first equation by $P^T$  and the second equation by $R^T$, we have
$$ P^T B R = P^T P V = N V; ~~~ R^TB^T P = R^T R W = M W.$$
Note that the left-hand side of one equation is the transposition of the other, so $ N V =W^T M$. On the other hand, these equations mean that matrices $P$ and $R$ define a partition of the incidence matrix $B$ into blocks $A_{i,j}$, where $A_{i,j}$ has sizes $n_i \times m_j$, row sums $v_{i,j}$,  and column sums $w_{j,i}$.

(2). Equality $N V = W^T M$ means that $n_i v_{i,j} = w_{j,i} m_j$ for all $i$ and $j$. Choose a nonnegative integer $t$ so that for all $i$ and $j$ there exist  $(0,1)$-matrices  $A_{i,j}$ of sizes $tn_i \times tm_j$ such that each row of $A_{i,j}$ contains exactly $v_{i,j}$ ones and each column contains exactly $w_{j,i}$ ones.  

Let a rectangular matrix  $B $ is given by  the block matrix $ \{ A_{i,j}\}$. We treat $B$  as  the incidence matrix of a hypergraph $\mathcal{G}$.
It can be verified directly that a $k$-coloring of $\mathcal{G}$
$$P = \left(  \begin{array}{ccc}
1  & \cdots & 0  \\
\vdots  & \vdots & \vdots \\
1  & \cdots & 0  \\
\vdots & \vdots & \vdots \\ 
0  & \cdots & 1  \\
\vdots & \vdots & \vdots \\
0 & \cdots & 1  \\
\end{array} \right),$$
where the $i$-th column contains exactly $tn_i$ ones, is a perfect coloring  with the incidence parameters $(V,W)$.  
\end{proof}

At last, we define refinements of perfect colorings in hypergraphs  similar to graphs.   We will say that a  coloring $f$ of a hypergraph $\mathcal{G}$  is a \textit{refinement} of a coloring $g$ if the coloring of the incidence graph of $\mathcal{G}$ induced by $f$ is a refinement of the  coloring induced by $g$.  The set of perfect colorings of a hypergraph with respect to the refinement forms a partially ordered set. 
Theorem~\ref{WLgraph} implies the following.

\begin{teorema} \label{WLhyper}
Given a hypergraph $\mathcal{G} = (X,E)$, there is a perfect coloring $f$ of $\mathcal{G}$  such that any other perfect coloring is a refinement of $f$. Moreover, for every coloring $g$ of $\mathcal{G}$ there is a refinement $h$ of $g$ such that $h$ is a perfect coloring and any other perfect coloring that refines $g$ is a refinement of $h$. 
\end{teorema}

\begin{proof}
To construct the coloring $f$, it is sufficient to apply Theorem~\ref{WLgraph} to the trivial bipartite coloring of the incidence graph $G$, in which the part $X$ is colored by one color, and  the part $E$ is colored by another  color.

The second statement of the theorem is obtained by the application of Theorem~\ref{WLgraph} to the coloring of the incidence graph induced by the coloring $g$. 
\end{proof}

We will say that the coloring $f$ from Theorem~\ref{WLhyper} is the \textit{minimal perfect coloring} of the hypergraph $\mathcal{G}$.

\subsection{The multidimensional matrix equation for perfect colorings}

 Let $\mathcal{G} (X,E)$ be a $d$-uniform hypergraph on $n$ vertices, $\mathbb{A}$ be the $d$-dimensional adjacency matrix of $\mathcal{G}$, and $P$ be a coloring of $\mathcal{G}$ into $k$ colors (color matrix of vertices).  We will say that  $P$    is \textit{perfect}   if  there exists a $d$-dimensional matrix $\mathbb{S}$ of order $k$  such that 
$$\mathbb{A} \circ  P = P \circ \mathbb{S},$$
where the product $\circ$ of multidimensional matrices is defined in Section~\ref{prelim-section}.
The matrix $\mathbb{S}$ is called the \textit{parameter matrix} of the perfect coloring $P$.

Let us show that this definition of perfect colorings in case of uniform hypergraphs is equivalent to the previous one. Moreover, we express entries of the parameter matrix $\mathbb{S}$ by the incidence parameters $(V,W)$.

\begin{teorema} \label{adjparam}
Let $\mathcal{G}$ be $d$-uniform hypergraph with the $d$-dimensional adjacency matrix $\mathbb{A}$. Then a coloring  $P$ of $\mathcal{G}$  into $k$ colors is a perfect  if and only if $\mathbb{A}  \circ P = P \circ \mathbb{S} $. Moreover,  the entries $s_{\gamma}$, $\gamma = (\gamma_1, \gamma_2, \ldots, \gamma_{d} )$, of $\mathbb{S}$ are
$$s_{\gamma}  = v_{\gamma_1,\gamma} \cdot {d-1 \choose  d_1,  \ldots, d_k}^{-1},$$
where $v_{\gamma_1,\gamma} $ is the number of hyperedges of color range $\gamma$  incident to a vertex of color $\gamma_1$, every color $l$ appears in the multiset $ \{\gamma_2, \ldots, \gamma_d \}$ exactly $d_l$ times, and  ${d-1 \choose  d_1,  \ldots, d_k}$ is the multinomial coefficient.
\end{teorema}

\begin{proof}
$\Leftarrow$: Suppose that $f$  is a $k$-coloring of the hypergraph $\mathcal{G}$ with a color matrix $P$ such that 
$$\mathbb{A}  \circ P = P \circ \mathbb{S}. $$

Let us consider entries of the matrix $\mathbb{C} =\mathbb{A}  \circ P $. Using the definition of the multidimensional matrix product,  for a given vertex $x$ and colors $j_1, \ldots, j_{d-1}$ we have 
\begin{equation} \label{c-entries}
c_{x, j_1, \ldots, j_{d-1}} = \sum\limits_{x_1, \ldots, x_{d-1} = 1}^n  a_{x, x_1, \ldots, x_{d-1}} p_{x_1, j_1} \cdots p_{x_{d-1}, j_{d-1}}.
\end{equation}

Note that  an entry $a_{x, x_1, \ldots, x_{d-1}}$ of $\mathbb{A}$ is equal to $\frac{1}{(d-1)!}$ if and only if $x, x_1, \ldots, x_{d-1}$ are  different vertices of $\mathcal{G}$ and constitute a hyperedge, and the entries $p_{x_i,j_i} = 1$ if and only if $f(x_i) = j_i$. 
Let $t_{x, j_1, \ldots, j_{d-1}}$ be the number of hyperedges $(x, x_1, \ldots, x_{d-1})$ such that  $\{f(x_1), \ldots, f(x_{d-1})\}$ and $\{ j_1, \ldots, j_{d-1}\}$ coincide as multisets. We also assume that each color $l$ appears in these multisets exactly $d_l$ times, $\sum\limits_{l=1}^k d_l = d-1$.  Then equation~(\ref{c-entries}) is equivalent to $c_{x, j_1, \ldots, j_{d-1}}  = t_{x, j_1, \ldots, j_{d-1}} {d-1 \choose  d_1, \ldots, d_k}^{-1}$, because every ordering of vertices  $x_1, \ldots, x_{d-1}$ that does not change the sequence of colors $f(x_1), \ldots, f(x_{d-1})$ gives a summand  $\frac{1}{(d-1)!}$  in counting of $c_{x, j_1, \ldots, j_{d-1}} $.

On the other hand, using $\mathbb{C} = P \circ \mathbb{S}$ and again the definition of $\circ$-operation, we have
$$c_{x, j_1, \ldots, j_{d-1}} = \sum\limits_{i=1}^k p_{x,i} s_{i, j_1, \ldots, j_{d-1}}.$$
Since $p_{x,i} =1$ if and only if $f(x) = i$, we have that if a vertex $x$ has color $i$, then $c_{x, j_1, \ldots, j_{d-1}} = s_{i, j_1, \ldots, j_{d-1}}$ and, consequently,  $s_{i, j_1, \ldots, j_{d-1}} =t_{x, j_1, \ldots, j_{d-1}} {d-1 \choose  d_1, \ldots, d_k}^{-1} $. Therefore, the color ranges $\{i,  j_1, \ldots, j_{d-1} \} $ of  hyperedges incident to a vertex $x$ of color $i$ are uniquely defined by entries of $\mathbb{S}$, and so the coloring $P$ is perfect.  Replacing $v_{i,(i, j_1, \ldots, j_{d-1})} = t_{x, j_1, \ldots, j_{d-1}}$ for all vertices $x$ of color $i$, we have the statement of the theorem. 

$\Rightarrow$: Assume that a coloring $P$ is perfect.  Repeating the calculation of entries of the matrix $\mathbb{A} \circ P$  and using the fact that the values $t_{x, j_1, \ldots, j_{d-1}}$ are uniquely defined by the color of vertex $x$, we see that there exists the required $d$-dimensional matrix $\mathbb{S}$ satisfying $\mathbb{A} \circ P = P \circ \mathbb{S}$.

\end{proof}

 Note that  a perfect coloring $P$ defines only the factor  $ v_{\gamma_1,\gamma}$ in the  entries $s_{\gamma}$ of the matrix $\mathbb{S}$, while the factor ${d-1 \choose  d_1,  \ldots, d_k}^{-1}$ in $s_{\gamma}$ depends  only on the index $\gamma$ and it is the same for all $d$-dimensional parameter matrices $\mathbb{S}$.

\begin{sled} \label{parammultii2}
Let $\mathcal{G}$ be a $d$-uniform  hypergraph, $P$ be a perfect coloring of $\mathcal{G}$,  $n_l$ be the number of vertices of color $l$, and $m_\gamma$ be the number of hyperedges of a color range $\gamma $. Suppose that $W = (w_{\gamma,l})$ is the EX-parameter matrix  and $\mathbb{S}$ is the multidimensional parameter matrix of the perfect coloring $P$.  Then entries $s_\gamma$ of  $\mathbb{S}$ are 
$$s_\gamma  = d \cdot \frac{m_{\gamma}}{n_{\gamma_1}}  \cdot    {d \choose  w_{\gamma,1},  \ldots, w_{\gamma,k}}^{-1}.$$
\end{sled}

\begin{proof}
By Theorem~\ref{adjparam}, the entries $s_{\gamma}$ of the parameter matrix $\mathbb{S}$ are equal to  $v_{\gamma_1,\gamma}  \cdot {d-1 \choose  d_1,  \ldots, d_k}^{-1}$, where $v_{\gamma_1,\gamma}$ are  entries of the $XE$-parameter matrix $V$  and each color $l$ appears in $\{\gamma_2, \ldots, \gamma_{d} \}$ exactly $d_{l}$ times.  By the definition of incidence parameters,  $\gamma = \{ \gamma_1, \ldots, \gamma_d \}$ contains the $l$-th color exactly $w_{\gamma,l}$ times, so $d_{\gamma_1} = w_{\gamma, \gamma_1} - 1$ and $d_{\gamma_i} = w_{\gamma, \gamma_i}$ for all other $\gamma_i$.

By Theorem~\ref{VWblock}(1), $n_{\gamma_1} v_{\gamma_1,\gamma}  = w_{\gamma, \gamma_1} m_{\gamma} $. Consequently,
$$s_{\gamma} = v_{\gamma_1,\gamma}   \cdot  \frac{d_1! \cdots d_k!}{(d-1)!}  =  \frac{v_{\gamma_1,\gamma}}{w_{\gamma, \gamma_1}} \cdot  \frac{w_{\gamma, 1}! \cdots w_{\gamma, k}!}{(d-1)!}  =  d \cdot  \frac{m_{\gamma}}{n_{\gamma_1}}  \cdot    {d \choose  w_{\gamma,1},  \ldots, w_{\gamma,k}}^{-1}.
$$

\end{proof}

\textbf{Example 2.}  Let $\mathcal{G}$ be the $3$-uniform hypergraph on $6$ vertices from Example~1 and $P$ be its perfect coloring:  
$$ P = 
\left(  \begin{array}{cc}
1 & 0 \\ 1 & 0 \\ 1 & 0 \\
0 & 1 \\ 0 & 1 \\ 0 & 1
\end{array} \right)$$
 Recall that the incidence parameters $(V,W)$ for this coloring are
 $$V = \left( 
\begin{array}{cc}
1 & 1 \\ 0 & 2 
\end{array}
\right); ~~~ 
W = 
\left( \begin{array}{cc}
3 & 0 \\ 1 & 2 
\end{array} \right).$$

The adjacency matrix $\mathbb{A}$ of the hypergraph $\mathcal{G}$ is the following $3$-dimensional matrix of order $6$:
\begin{gather*}
\mathbb{A} =  \frac{1}{2}\left(
\begin{array}{cccccc|cccccc|cccccc|}
0 & 0 & 0 & 0 & 0 & 0  &  0 & 0 & 1 & 0 & 0 & 0  &  0 & 1 & 0 & 0 & 0 & 0  \\
0 & 0 & 1 & 0 & 0 & 0  &  0 & 0 & 0 & 0 & 0 & 0  &  1 & 0 & 0 & 0 & 0 & 0  \\
0 & 1 & 0 & 0 & 0 & 0  &  1 & 0 & 0 & 0 & 0 & 0  &  0 & 0 & 0 & 0 & 0 & 0  \\
0 & 0 & 0 & 0 & 1 & 0  &  0 & 0 & 0 & 0 & 0 & 1  &  0 & 0 & 0 & 0 & 0 & 0  \\
0 & 0 & 0 & 1 & 0 & 0  &  0 & 0 & 0 & 0 & 0 & 0  &  0 & 0 & 0 & 0 & 0 & 1  \\
0 & 0 & 0 & 0 & 0 & 0  &  0 & 0 & 0 & 1 & 0 & 0  &  0 & 0 & 0 & 0 & 1 & 0  \\
\end{array}  \right. \\
\left. 
\begin{array}{|cccccc|cccccc|cccccc}
0 & 0 & 0 & 0 & 1 & 0  &  0 & 0 & 0 & 1 & 0 & 0  &  0 & 0 & 0 & 0 & 0 & 0 \\
0 & 0 & 0 & 0 & 0 & 1  &  0 & 0 & 0 & 0 & 0 & 0  &  0 & 0 & 0 & 1 & 0 & 0 \\
0 & 0 & 0 & 0 & 0 & 0  &  0 & 0 & 0 & 0 & 0 & 1  &  0 & 0 & 0 & 0 & 1 & 0  \\
0 & 0 & 0 & 0 & 0 & 0  &  1 & 0 & 0 & 0 & 0 & 0  &  0 & 1 & 0 & 0 & 0 & 0  \\
1 & 0 & 0 & 0 & 0 & 0  &  0 & 0 & 0 & 0 & 0 & 0  &  0 & 0 & 1 & 0 & 0 & 0  \\
0 & 1 & 0 & 0 & 0 & 0  &  0 & 0 & 1 & 0 & 0 & 0  &  0 & 0 & 0 & 0 & 0 & 0  \\
\end{array}
\right).
\end{gather*}
It can be checked directly that $\mathbb{A} \circ P = P \circ \mathbb{S}$ for the $3$-dimensional parameter matrix
$$\mathbb{S} = 
\left( 
\begin{array}{cc|cc}
1 & 0 & 0 & 1 \\
0 & 1  & 1 &  0 
\end{array}
\right), $$
where $s_{1,1,1} = v_{1,1}  / 1= 1 $, $s_{1, 2,2} = v_{1,2} / 1 = 1$, $s_{2, 2, 1} = s_{2, 1,2} = v_{2,2} / 2 = 1$. Note that $s_{1,1,2} = s_{1,2,1}  = s_{2,1,1}= s_{2,2,2} = 0$ because there are no hyperedges of color ranges $\{ 1,1,2\}$ and $\{ 2,2,2\}$ in the perfect coloring $P$.
\bigskip

Using the obtained results, let us prove several properties of the parameter matrix $\mathbb{S}$. 

\begin{utv}
Let $P$ be a perfect coloring in a $d$-uniform hypergraph $\mathcal{G}$ with the parameter matrix $\mathbb{S}$. Then the sum of entries of $\mathbb{S}$ along the $i$-th hyperplane of the first direction is equal to the degree of vertices of color $i$.
\end{utv}

\begin{proof}
By Theorem~\ref{adjparam}, the entries $s_{\gamma}$ of the parameter matrix $\mathbb{S}$ are equal to  $v_{\gamma_1,\gamma}   \cdot {d-1 \choose  d_1,  \ldots, d_k}^{-1}$, where  $v_{\gamma_1,\gamma} $ is the number of hyperedges of color range $\gamma$  incident to a vertex of color $\gamma_1$. 
Note that every hyperedge with color range $\gamma $ gives exactly ${d-1 \choose  d_1, \ldots, d_k}$ nonzero entries $s_{\gamma}$ in the $\gamma_1$-th hyperplane of the first direction  for the parameter matrix $\mathbb{S}$. So the total contribution of  such hyperedges   to this sum is $v_{\gamma_1,\gamma}$, and the total sum of entries along this hyperplane is the degree of a vertex of color $\gamma_1$. 
\end{proof}

\begin{teorema} \label{symmetrparam}
Let $\mathcal{G}$ be a $d$-uniform hypergraph. Assume that $P$ is a perfect coloring of $\mathcal{G}$ with the parameter matrix $\mathbb{S}$ and $N = P^T P$. Then the matrix $\mathbb{H} = N  \circ \mathbb{S}$ is symmetric.
\end{teorema}

\begin{proof}
Recall that $N= P^TP$ is a diagonal matrix whose $i$-th diagonal entry is the number $n_i$ of vertices of color $i$ in the coloring $P$. 

By Corollary~\ref{parammultii2}, an entry $s_\gamma $ of the matrix $\mathbb{S}$ is equal to $d \cdot  \frac{m_{\gamma}}{n_{\gamma_1}}   \cdot  {d \choose  w_{\gamma,1},  \ldots, w_{\gamma,k}}^{-1}$, where  $w_{\gamma,l}$ is the number of appearances of symbol $l$  in $\gamma$ and $m_{\gamma}$ is the number of hyperedges of $\mathcal{G}$ having the color range $\gamma$.

Using the definition of the product of multidimensional matrices, we see that  $\mathbb{H} = N \circ \mathbb{S}$ is a $d$-dimensional matrix with entries
$$h_{\beta} =   \sum\limits_{j=1}^k n_{\beta_1, j} \cdot  s_{j, \beta_2, \ldots, \beta_d} =  d \cdot  m_{\beta}  \cdot   {d \choose  w_{\beta,1},  \ldots, w_{\beta,k}}^{-1}. $$

Since for every permutation $\sigma \in S_d$ it holds $h_{\beta} = h_{\sigma(\beta)}$, we have that the matrix $\mathbb{H}$ is symmetric. 

 \end{proof}

It is well known that for graphs the spectrum of the parameter matrix  of a perfect coloring is contained in the spectrum of the adjacency matrix. A similar fact (regarding geometric multiplicities) is true for hypergraphs. Moreover, this property holds not only for perfect colorings of hypergraphs, but for general matrices satisfying the same equation.  To prove this statement, we need the following auxiliary  result.

\begin{utv} \label{idtransit}
Let $P$ be a color matrix of size $n \times k$ and $\mathbb{I}$ be the $d$-dimensional identity matrix. Then 
$$\mathbb{I} \circ P = P \circ \mathbb{I}.$$
\end{utv}

\begin{proof}
The proof is based on the definitions of $\circ$-operation  and  matrices $\mathbb{I}$ and $P$.

Let us show that $(\alpha_1, \alpha_2, \ldots, \alpha_d)$-th entries of matrices in left-hand and right-hand sides of the equation $\mathbb{I} \circ P = P \circ \mathbb{I}$ are equal to $1$ if $\alpha_2 = \cdots = \alpha_d$ and $p_{\alpha_1, \alpha_2} = 1$ and they are $0$ otherwise. 

 Indeed,  an $(\alpha_1, \alpha_2, \ldots, \alpha_d)$-th entry of the matrix $\mathbb{I} \circ P$ is
$$\sum\limits_{\beta_2, \ldots, \beta_d =1}^n i_{\alpha_1, \beta_2, \ldots, \beta_d} \cdot p_{\beta_2, \alpha_2} \cdots p_{\beta_d, \alpha_d}.$$
Note that $i_{\alpha_1, \beta_2, \ldots, \beta_d} = 1$ if and only if $\alpha_1 = \beta_2 = \cdots = \beta_d$. Otherwise, $i_{\alpha_1, \beta_2, \ldots, \beta_d}  = 0$.  By the definition of a color matrix, there is a unique $j$ such that $p_{\alpha_1,j} = 1$. 

On the other hand, an $(\alpha_1, \alpha_2, \ldots, \alpha_d)$-th entry of the matrix $P \circ \mathbb{I}$ is
$$\sum\limits_{j =1}^k p_{\alpha_1, j} \cdot i_{j, \alpha_2, \ldots, \alpha_d} .$$
Again, there is a unique $j$ such that $p_{\alpha_1, j} = 1$ and $i_{j, \alpha_2, \ldots, \alpha_d} = 1$ if and only if $j = \alpha_2 = \cdots = \alpha_d$. 
\end{proof}

\begin{teorema} \label{specinclude}
Let $\mathbb{A}$ be a $d$-dimensional matrix of order $n$, $P$ be a color matrix of size $n \times k$, and $\mathbb{B}$ be a $d$-dimensional matrix of order $k$ such that $\mathbb{A} \circ P = P \circ \mathbb{B}$. If $(\lambda, x)$ is an eigenpair for the matrix $\mathbb{B}$, then $(\lambda, Px)$ is an eigenpair for  $\mathbb{A}$. 
\end{teorema}

\begin{proof}
By the definition, $(\lambda, x)$ is an eigenpair for the matrix $\mathbb{B}$ if and only if 
$\mathbb{B}\circ x  = \lambda (\mathbb{I} \circ x ).$
Using Proposition~\ref{prodcomute} and the equality $ P \circ \mathbb{I} = \mathbb{I} \circ P$ from Proposition~\ref{idtransit}, we conclude that
\begin{gather*}
\mathbb{A} \circ (Px) = \mathbb{A} \circ (P \circ x) =  (\mathbb{A}\circ P) \circ x  = (P \circ \mathbb{B}) \circ x = P \circ (\mathbb{B}\circ x)  = \\ P \circ (\lambda \mathbb{I} \circ x)  = \lambda (P \circ \mathbb{I}) \circ x =  \lambda (\mathbb{I} \circ P) \circ x =   \lambda \mathbb{I} \circ (P \circ x) = \lambda \mathbb{I} \circ (Px).
\end{gather*}
\end{proof}

This theorem easily implies that all eigenvalues of the parameter matrix of a perfect coloring  are also eigenvalues of the adjacency matrix of a hypergraph. 

\begin{teorema} \label{adjcpecinclude}
Let $\mathcal{G}$ be a $d$-uniform  hypergraph with the adjacency matrix $\mathbb{A}$ and $P$ be a perfect coloring with the $d$-dimensional parameter matrix $\mathbb{S}$. If $(\lambda, x)$ is an eigenpair for the matrix $\mathbb{S}$, then $(\lambda, Px)$ is an eigenpair for the matrix $\mathbb{A}$. 
\end{teorema}

\begin{proof}
It is sufficient to note that  $\mathbb{A} \circ P = P  \circ \mathbb{S}$ and apply Theorem~\ref{specinclude}.
\end{proof}

\begin{sled} \label{sumdegree}
Assume that $P$ is a perfect coloring of a $d$-uniform hypergraph $\mathcal{G}$ into $k$ colors with the parameter matrix $\mathbb{S}$ and $\lambda$ is an eigenvalue of $\mathbb{S}$. Then there is an eigenvector $y$ of $\mathcal{G}$ for eigenvalue $\lambda$ such that  components of $y$ attain at most $k$ different values.
\end{sled}

\begin{proof}
Let $x$ be an eigenvector of $\mathbb{S}$ for the eigenvalue $\lambda$.  Set $y = Px$ and repeat the proof of Theorem~\ref{specinclude}.
\end{proof}

\section{Coverings of hypergraphs} \label{covering-setion}

We will say that a hypergraph $\mathcal{G}$ \textit{covers} a hypergraph $\mathcal{H}$ if there is a map (\textit{covering}) $\varphi: X(\mathcal{G}) \rightarrow X(\mathcal{H})$ such that for each hyperedge $e \in E(\mathcal{G})$  the set $\{ \varphi (x) | x \in e \}$ is a hyperedge of $\mathcal{H}$ and $\varphi$ saves a collection of incident hyperedges for each vertex.  

Note that a covering preserves degrees of vertices and sizes of hyperedges, so every covering of a $d$-uniform $r$-regular hypergraph is also $d$-uniform and $r$-regular. 

Every covering $\varphi$ of a hypergraph $\mathcal{H}$ by a hypergraph $\mathcal{G}$ may be considered as a coloring of $\mathcal{G}$ by vertices of $\mathcal{H}$. Moreover,   $\varphi$ colors the hyperedges of $\mathcal{G}$ by colors corresponding  to hyperedges of $\mathcal{H}$. It gives us the following statement.

\begin{utv} \label{colorcover}
A map $\varphi: X(\mathcal{G}) \rightarrow X(\mathcal{H})$  is a covering of a hypergraph $\mathcal{H}$ by a hypergraph $\mathcal{G}$ if and only if $\varphi$ is a perfect coloring of $\mathcal{G}$ with the parameter matrix $\mathbb{S}$ equal to the adjacency matrix $\mathbb{A}_H$ of $\mathcal{H}$. The incidence parameters of this coloring are $(C, C^T)$, where $C$  is the incidence matrix of $\mathcal{H}$. 
\end{utv}

An equivalent  definition of coverings of hypergraphs was proposed in~\cite{LiHou.hypercover}.   In particular, in~\cite[Theorem 7]{LiHou.hypercover} it was  stated that  hypergraph coverings are   equivalent to coverings of their incidence graphs.  It is also straightforward from our notions and Proposition~\ref{colorcover}. 

\begin{utv} \label{hyperbicolor}
A coloring $\varphi$ is a covering of a hypergraph $\mathcal{H}$  by a hypergraph  $\mathcal{G}$ if and only if the induced bipartite perfect coloring $\overline{\varphi}$ of the incidence graph $G$ is a covering of the incidence graph $H$. 
\end{utv}

We will say that a covering $\varphi$ of a hypergraph $\mathcal{H}$ by $\mathcal{G}$  is a \textit{$k$-covering}  if for every $y \in X(\mathcal{H})$ there are exactly $k$ vertices $x \in X(\mathcal{G})$ such that $\varphi(x) = y$.   Proposition~\ref{hyperbicolor} implies that  every covering of a hypergraph is a $k$-covering.  

\begin{utv} \label{kcover}
For every  covering $\varphi$ of a  hypergraph $\mathcal{H} = (Y,U)$ by $\mathcal{G} = (X,E)$ there is $k\in \mathbb{N}$ such that $\varphi$ is $k$-covering. In particular,  we have $|X| = k|Y|$ and $|E| = k |U|$.
\end{utv}

At last, we note that $k$-coverings of hypergraphs can be nicely described  in terms of incidence matrices.

\begin{teorema} \label{coverinci}
Let $\mathcal{G}$ and $\mathcal{H}$ be hypergraphs with incidence matrices $B$ and $C$,  respectively. If there is a $k$-covering of $\mathcal{H}$  by $\mathcal{G}$, then the matrix $B$ is a block matrix $\{ D_{i,j}\}$, where $D_{i,j}$ is a permutation matrix of order $k$ if $c_{i,j} = 1$ and $D_{i,j}$ is the zero matrix of order $k$ otherwise. 
\end{teorema}

\begin{proof}
By Proposition~\ref{colorcover}, we can consider a $k$-covering $\varphi$ as a perfect coloring of $\mathcal{G}$ with the incidence parameters $(C, C^T)$.  In particular, we have that the size of each color class in $\varphi$ is $k$. To prove the theorem, it only remains to apply Theorem~\ref{VWblock}(1) to the coloring $\varphi$.
\end{proof}

\textbf{Example 3.} Let $\mathcal{G}$ be the $3$-uniform hypergraph  with the vertex set $X = \{ x_1, x_2, \ldots , x_8\}$ and eight hyperedges $e_1 = \{ x_3, x_6, x_7\}$, $e_2 = \{ x_4, x_5, x_8\}$, $e_3 = \{ x_1, x_6, x_7\}$, $e_4 = \{ x_2, x_5, x_8\}$, $e_5 = \{ x_1, x_4, x_8\}$,  $e_6 = \{ x_2, x_3, x_7\}$, $e_7 = \{ x_2, x_3, x_5\}$, $e_8 = \{ x_1, x_4, x_6\}$, and $\mathcal{H}$ be the $3$-uniform hypergraph  with the vertex set $X' = \{ x'_1, x'_2, x'_3, x'_4\}$ and four hyperedges $e'_1 = \{ x'_2, x'_3, x'_4\}$,  $e'_2 = \{ x'_1, x'_3, x'_4\}$,  $e'_3 = \{ x'_1, x'_2, x'_4\}$,  $e'_4 = \{ x'_1, x'_2, x'_3\}$. 

Then the map $\varphi : X \rightarrow X'$ such that $\varphi(x_1) = \varphi(x_2) = x'_1$, $\varphi(x_3) = \varphi(x_4) = x'_2$, $\varphi(x_5) = \varphi(x_6) = x'_3$, $\varphi(x_7) = \varphi(x_8) = x'_4$ is a $2$-covering of the hypergraph $\mathcal{H}$ by the hypergraph $\mathcal{G}$. In particular,  $\varphi(e_1) = \varphi(e_2)  = e'_1$, $\varphi(e_3) = \varphi(e_4)  = e'_2$, $\varphi(e_5) = \varphi(e_6)  = e'_3$,  and $\varphi(e_7) = \varphi(e_8)  = e'_4$.  

The incidence matrices $B$ and $C$ of the hypergraphs $\mathcal{G}$ and $\mathcal{H}$, respectively, are
$$
B = \left( \begin{array}{cccccccc}
0 & 0 & 1 & 0 & 1 & 0 & 0 & 1 \\
0 & 0 & 0 & 1 & 0 & 1 & 1 & 0 \\
1 & 0 & 0 & 0 & 0 & 1 & 1 & 0 \\
0 & 1 & 0 & 0 & 1 & 0 & 0 & 1 \\
0 & 1 & 0 & 1 & 0 & 0 & 1 & 0 \\
1 & 0 & 1 & 0 & 0 & 0 & 0 & 1 \\
1 & 0 & 1 & 0 & 0 & 1 & 0 & 0 \\
0 & 1 & 0 & 1 & 1 & 0 & 0 & 0 \\
\end{array}
\right); ~~~
C = \left(  \begin{array}{cccc}
0 & 1 & 1 & 1 \\
1 & 0 & 1 & 1 \\
1 & 1 & 0 & 1 \\
1 & 1 & 1 & 0 \\
\end{array}
\right).
$$
It can checked that they satisfy the condition of Theorem~\ref{coverinci}.
\bigskip

The correspondence between coverings and colorings allows us to establish several properties of coverings of hypergraphs. First of all, let us show that if a uniform hypergraph $\mathcal{G}$ covers a hypergraph $\mathcal{H}$, then eigenvalues of $\mathcal{H}$ are also eigenvalues of $\mathcal{G}$.  This result was recently obtained in~\cite[Theorem 4.1]{SongFanWang.hyperspeccover}.

\begin{teorema} \label{coverspecinclude}
Assume that a $d$-uniform hypergraph $\mathcal{G}$ with the adjacency matrix $\mathbb{A}_G$ covers a hypergraph $\mathcal{H}$ with the adjacency matrix $\mathbb{A}_H$. Then every eigenvalue of $\mathbb{A}_H$ is an eigenvalue of $\mathbb{A}_G$ (regarding geometric multiplicities).
 \end{teorema}

\begin{proof}
By the definition of coverings, there is a perfect coloring $P$ of $\mathcal{G}$ with the parameter matrix $\mathbb{A}_H$. By Theorem~\ref{adjcpecinclude}, every eigenvector of the matrix $\mathbb{A}_H$ for an eigenvalue $\lambda$ gives an eigenvector of the matrix $\mathbb{A}_G$ for the same eigenvalue. 
\end{proof}

Next, let us prove that a covering preserves the set of perfect colorings of hypergraphs. 

\begin{teorema} \label{covercolorinclude}
Assume that a $d$-uniform hypergraph $\mathcal{G}$ with the adjacency matrix $\mathbb{A}_G$ covers a hypergraph $\mathcal{H}$ with the adjacency matrix $\mathbb{A}_H$. Then for every perfect coloring of $\mathcal{H}$ with the parameter matrix $\mathbb{S}$ there is a perfect coloring of $\mathcal{G}$ with the same parameter matrix $\mathbb{S}$.
\end{teorema}

\begin{proof}
Let $P$ be a perfect coloring of $\mathcal{H}$ with the parameter matrix $\mathbb{S}$:
$$\mathbb{A}_H \circ P = P \circ \mathbb{S}.$$
Since $\mathcal{G}$ covers $\mathcal{H}$, there is a color matrix $R$ such that
$$\mathbb{A}_G \circ R = R \circ \mathbb{A}_H. $$
Then, using properties of the product of multidimensional matrices  from Proposition~\ref{prodcomute}, we have
$$\mathbb{A}_G \circ (RP)  = (\mathbb{A}_G \circ R) \circ P   = (R \circ \mathbb{A}_H) \circ P = R \circ (\mathbb{A}_H \circ P) = R \circ (P \circ \mathbb{S}) = RP \circ \mathbb{S}.$$
Consequently, $RP$ is a perfect coloring of $\mathcal{G}$ with the parameter matrix $\mathbb{S}$. 
\end{proof}

At last, we establish a result on the existence of common coverings for  hypergraphs with several corollaries.  They can be obtained by an accurate combination of the theorem on the existence of common coverings for graphs (see~\cite{leighton.comcover}) and  the correspondence between coverings of hypergraphs and their incidence graphs (Proposition~\ref{hyperbicolor}). Meanwhile, here we provide a direct proof of this result, reformulating the ideas from~\cite{leighton.comcover} in terms of incidence matrices. For future applications, we state this result for multihypergraphs.

\begin{teorema} \label{commoncover}
If connected  (multi)hypergraphs $\mathcal{H}$ and $\mathcal{H}'$ have perfect colorings with the same incidence parameters $(V,W)$, then there is a hypergraph $\mathcal{G}$ that covers both $\mathcal{H}$ and $\mathcal{H}'$.
\end{teorema}

\begin{proof}
Assume that $C$ and  $C'$ are the incidence matrices  of $\mathcal{H}$ and $\mathcal{H}'$  and  $\varphi$ and $\varphi'$ are their perfect colorings with incidence parameters $(V,W)$, respectively.  Let $n_i$ and $m_j$ ($n_i'$ and $m_j'$), $i = 1, \ldots, k$, $j= 1, \ldots, l$ be the number of vertices of color $i$ and the number of hyperedges of color $j$ in the induced bipartite perfect coloring of $\mathcal{H}$ (of $\mathcal{H}'$). By Theorem~\ref{VWblock}(1),  there are numbers $\pi_{i,j}$ and  $\pi'_{i,j}$ such that
$$\pi_{i,j} = n_i v_{i,j} = w_{j,i} m_j; ~~~ \pi'_{i,j} =  n'_i v_{i,j} = w_{j,i} m'_j.$$

Moreover, by Theorem~\ref{VWblock}(1), the incidence matrix $C$ is  the block matrix $\{ A_{i,j}\}$ such that $A_{i,j}$ is a $(0,1)$-matrix of size $n_i \times m_j$ with row sums $v_{i,j}$ and column sums $w_{j,i}$. In particular, the block $A_{i,j}$ contains $\pi_{i,j} $ ones.  Let us index all nonzero entries $a_{x,e}$ of $A_{i,j}$ by pairs $(\mu, \nu)$, $\mu \in \{ 0, \ldots, w_{j,i} -1 \}$, $\nu \in \{ 0,\ldots, v_{i,j} -1\}$, so that $\mu(x)$ is the position of $a_{x,e}$ in the $e$-th column of $A_{i,j}$ with respect to other nonzero entries (from up to down), and $\nu(e)$ is the position of this entry in the $x$-th row of $A_{i,j}$ (from left to right).

Similarly, the incidence matrix $C'$ is a block matrix $\{ A'_{i,j}\}$ with blocks $A'_{i,j}$ of sizes $n'_i \times m'_j$, row sums $v_{i,j}$, and column sums $w_{j,i}$. We also label all nonzero entries of $\{ A'_{i,j}\}$ by pairs $(\mu', \nu')$ in the same way.

Let $\Pi$ denote the least common multiple of the numbers $\pi_{i,j}$:
$\Pi = \mathrm{lcm}_{i , j} (\pi_{i,j} ),$ and let $t_i = \Pi / n_i$ and $s_j = \Pi / m_j$. Note that all $t_i$ and $s_j$ are integer.

We construct the incidence matrix $B$ of a covering (multi)hypergraph $\mathcal{G}$ as a block matrix $\{ D_{i,j}\}$, $i =1, \ldots, k$, $j= 1, \ldots, l$.  The size of each block $D_{i,j}$ will be $t_{i} n_i n'_i \times s_j m_j m'_j $.   

 Note that the sizes of blocks $ D_{i,j} $  remain the same, if  we use  $t'_i$ and $s'_j$ instead $t_i$ and $s_j$. It is shown in the following claim.

\textbf{Claim:}  For all $i = 1, \ldots, k$, $j = 1, \ldots, l$, we have  that $t_{i}   = t'_{i} $ and  $s_j   = s'_j $, where $t_i = \Pi / n_i$, $t'_i = \Pi' / n'_i$, $s_j = \Pi / m_j$, $s'_j = \Pi / m'_j$, $\Pi = \mathrm{lcm} (\pi_{i,j})$, $\Pi' = \mathrm{lcm} (\pi'_{i,j})$. 

\begin{proof}[Proof of claim]
Thanks to symmetry, it is sufficient to prove  the equalities for all $t_i$ and $t'_i$. 
Without loss of generality, we show  that $t_{1}  = t'_{1} $, i.e., $\frac{\mathrm{lcm} (\pi_{i,j})}{n_1} = \frac{\mathrm{lcm} (\pi'_{i,j})}{n'_1}$.   Since $\mathcal{H}$ is a connected (multi)hypergraph, there are Berge paths from a vertex of color $1$ to any vertex of color $c \neq 1$. Using relations $n_i v_{i,j} = w_{j,i} m_j$ from Theorem~\ref{VWblock}(1) for  colors of vertices $i$ and colors of hyperedges $j$ along this path, we conclude that there are rational  numbers $r_{1,c}$ depending only $V$ and $W$ such that $n_{c} = r_{1,c} n_1 $ for every $c \geq 1$. Using these equalities, we conclude that there is a rational number $q$ depending only on $V$ and $W$ such that $\Pi = \mathrm{lcm} (\pi_{i,j}) = q n_1$.  

Acting similarly, in case of the  $\mathcal{H}'$ we get the same $q$  such that $\Pi' = \mathrm{lcm} (\pi'_{i,j}) = q n'_1$.  Consequently, $t_1 = t'_1$. 
\end{proof}

Let us describe the structure of blocks $D_{i,j}$. Given $i$ and $j$, each $D_{i,j}$  is a block matrix  $\{ F_{\alpha,\beta} \}$ with blocks of sizes $t_i \times s_j$,  each index $\alpha$ has a form $(x,x')$, $x \in X(\mathcal{H})$, $x' \in X(\mathcal{H}')$,  and an index $\beta$ has a form $(e,e')$, $e \in E(\mathcal{H})$, $e' \in E(\mathcal{H}')$. Each block $F= F_{\alpha, \beta}$, where $\alpha = (x,x')$, $\beta = (e,e')$,  is a $(0,1)$-matrix defined as follows.
\begin{itemize}
\item  Assume that  for entries of $A_{i,j}$  and $A'_{i,j}$ it holds $a_{x,e} = 1$ and $a'_{x',e'} = 1$.  If  $\sigma  \equiv \nu(e) - \nu'(e')  \mod v_{i,j}$  and  $\delta   \equiv \mu(x) - \mu'(x') \mod w_{j,i}$, then the $\sigma$-th row and the $\delta$-th column of the block $F_{\alpha, \beta}$ contain exactly one $1$, otherwise they are filled by $0$.
\item If we have $a_{x,e} = 0$ or  $a'_{x',e'} = 0$, then the block $F_{\alpha, \beta}$ is the zero matrix. 
\end{itemize}

Note that  the nonzero  blocks $F_{\alpha, \beta}$ contain exactly $t_i / v_{i,j} = \frac{\Pi}{n_i v_{i,j}}$ nonzero rows  that coincides with the number $s_j / w_{j,i} = \frac{\Pi}{m_j w_{j,i}}$ of nonzero columns. Since both of these numbers are integers,  the blocks $F_{\alpha, \beta}$ are well defined.

It only remains to show that the constructed $(0,1)$-matrix $B = \{ D_{i,j}\}$  (with $D_{i,j} = \{ F_{\alpha, \beta} \}$) is the incidence matrix of a (multi)hypergraph $\mathcal{G}$ covering both $\mathcal{H}$ and $\mathcal{H}'$. Without loss of generality, let us prove that $\mathcal{G}$ covers  $\mathcal{H}$.

For every $(x,e)$-entry of the incidence matrix $C$ of $\mathcal{H}$, consider a submatrix (block) $J_{x,e}$ of $B$ composed of  blocks $F_{\alpha, \beta}$ such that  $\alpha = (x,x')$, $\beta = (e,e')$, $(x',e')$ runs over all entries of $C'$.  It can be checked that  $J_{x,e}$ is a  square matrix of order $\Pi'$. 

Let  $x$ has color $i$ and $e$ has color $j$ in the perfect coloring $\varphi$ of $\mathcal{H}$. Then $J_{x,e}$ is a submatrix contained in the block $D_{i,j}$ of $B$.

The definition of blocks $F_{\alpha,\beta}$ implies that if  $(x,e)$-entry of $C$ is $0$, then the corresponding block $J_{x,e}$ is the zero matrix.

Suppose that $(x,e)$-entry of $C$ is  equal to $1$.  Then every row of the submatrix $J_{x,e}$ contains exactly one $1$, because, (by the construction of  $F_{\alpha, \beta}$) there is a unique $(x',e')$-entry of $C'$ that gives the block $F_{\alpha,\beta}$ having the unity entry exactly at this row. For similar reasons, every column of $J_{x,e}$ contains exactly one $1$. Consequently, $J_{x,e}$ is a permutation matrix.

By Theorem~\ref{coverinci}, it means that the constructed (multi)hypergraph $\mathcal{G}$ covers $\mathcal{H}$. If  $\mathcal{G}$  has  multiple hyperedges   (i.e. $B$ contains identical columns), then one can cover $\mathcal{G}$ by some hypergraph $\mathcal{G}'$ that also covers $\mathcal{H}$ and $\mathcal{H}'$. 
\end{proof}

\begin{zam}
 In definition of matrices $F_{\alpha,\beta}$ instead relations $ \sigma \equiv \nu(e) - \nu'(e')  \mod v_{i,j}$  and  $\delta \equiv \mu(x) - \mu'(x') \mod w_{j,i}$ we can use any other latin squares of orders $v_{i,j}$ and $w_{j,i}$, whose   entries we treat  as  triples $(\nu(e), \nu'(e'), \sigma)$ and $(\mu(x), \mu'(x'), \delta)$, respectively. 
 \end{zam}

Let us obtain several corollaries of  Theorem~\ref{commoncover}. Firstly,  we reformulate  it in terms of multidimensional matrices.

\begin{sled} \label{matrixcommoncover}
Let $\mathbb{A}$ and $\mathbb{A}'$ be $d$-dimensional adjacency matrices of $d$-uniform hypergraphs $\mathcal{H}$ and $\mathcal{H}'$. If there are color matrices $P$ and $P'$ and a $d$-dimensional matrix $\mathbb{S}$ such that $\mathbb{A} \circ P = P \circ \mathbb{S}$ and $\mathbb{A}' \circ P' = P' \circ \mathbb{S}$, then there  are color matrices $R$ and $R'$ and a $d$-dimensional matrix $\mathbb{L}$ such that $\mathbb{L} \circ R = R \circ \mathbb{A}$ and $\mathbb{L} \circ R' = R' \circ \mathbb{A}'$. 
\end{sled}

\begin{proof}
The statement directly follows from the existence of a common covering (Theorem~\ref{commoncover}), an interpretation of a covering as a perfect coloring (Proposition~\ref{colorcover}), and  the definition of perfect colorings by the means of multidimensional adjacency matrices (Theorem~\ref{adjparam}).
\end{proof}

\begin{sled} \label{Leightonhyper}
Let $\mathcal{H}$ and $\mathcal{H}'$ be $d$-uniform connected hypergraphs. There are perfect colorings of $\mathcal{H}$ and $\mathcal{H}'$ with the same parameter matrix $\mathbb{S}$ if and only if there exists a hypergraph $\mathcal{G}$ that covers both $\mathcal{H}$ and $\mathcal{H}'$. 
\end{sled}

\begin{proof}

Necessity is proved in Corollary~\ref{matrixcommoncover} and in Theorem~\ref{commoncover}. 

Let us prove sufficiency. Suppose that a hypergraph $\mathcal{G}$ covers hypergraphs  $\mathcal{H}$ and $\mathcal{H}'$ with adjacency matrices $\mathbb{A}$ and $\mathbb{A}'$, respectively. Assume that there are no perfect colorings of $\mathcal{H}$ and $\mathcal{H}'$ with the same parameter matrix $\mathbb{S}$. In particular, the minimal perfect colorings $P$ and $P'$ of these hypergraphs (that exist by Theorem~\ref{WLhyper}) have different parameter matrices  $\mathbb{S}$ and  $\mathbb{S}'$:
$$\mathbb{A} \circ P = P \circ \mathbb{S}; ~~ \mathbb{A}' \circ P' = P' \circ \mathbb{S}'.$$

By Theorem~\ref{covercolorinclude}, the hypergraph $\mathcal{G}$ has perfect colorings with the parameter matrices   $\mathbb{S}$ and  $\mathbb{S}'$. Using again Theorem~\ref{WLhyper}, find the minimal perfect coloring of the hypergraph $\mathcal{G}$. It has the parameter matrix $\mathbb{T}$ such that $\mathbb{T}$ is different from at least one of the matrices   $\mathbb{S}$ and  $\mathbb{S}'$. Without loss of generality, assume that $\mathbb{T}$ is not equal to  $\mathbb{S}$. Then the perfect coloring of $\mathcal{G}$ with the parameter matrix $\mathbb{S}$ is a refinement of the minimal perfect coloring of $\mathcal{G}$. Equivalently, there exists a color matrix $R$ such that $\mathbb{S} \circ R = R \circ \mathbb{T}$. Using Proposition~\ref{prodcomute}, we get
\begin{gather*}
\mathbb{A} \circ (P  R) = \mathbb{A} \circ (P \circ R) =  (\mathbb{A} \circ P) \circ R  = (P \circ  \mathbb{S})  \circ R =   \\
P \circ  (\mathbb{S}  \circ R) = P \circ (R \circ \mathbb{T}) = (P R ) \circ \mathbb{T}.
\end{gather*}
Therefore, $PR$  is a perfect coloring of $\mathcal{H}$ in which some color classes are a union of color classes of $P$: a contradiction with the minimality of the perfect coloring $P$.
\end{proof}

We also prove that every regular uniform hypergraph can be covered by a multipartite hypergraph that can be partitioned into perfect matchings.

\begin{sled}
For every $d$-uniform $r$-regular hypergraph $\mathcal{H}$ there is a  hypergraph $\mathcal{G}$  such that $\mathcal{G}$ covers $\mathcal{H}$ and $\mathcal{G}$  is $d$-partite hypergraph that can be partitioned into $r$ perfect matchings.
\end{sled}

\begin{proof}
Consider an auxiliary multihypergraph $\mathcal{F}$ with the incidence matrix $B$ of size $d \times r$ whose all entries equal to $1$. It consists of $d$ vertices $x_1, \ldots, x_d$ and a hyperedge $e = (x_1, \ldots, x_d)$ of multiplicity $r$. 
By the definition, each hyperedge of $\mathcal{F}$ is a perfect matching and $\mathcal{F}$ is a $d$-partite hypergraph (each vertex is a part). 

Colorings of all vertices into one color are perfect colorings of $\mathcal{H}$ and $\mathcal{F}$ that also induce monochrome colorings of hyperedges of these hypergraphs.  The incidence parameter matrices $W$ and $V$ of these colorings are the same, have size $1 \times 1$ and, by the definition, $W = (d)$ and $V = (r)$. By Theorem~\ref{commoncover}, there is a hypergraph $\mathcal{G}$ covering both $\mathcal{H}$ and $\mathcal{F}$. Since the covering preserves the incidence relations between vertices and hyperedges, we see that every vertex of $\mathcal{F}$ corresponds to a part of $\mathcal{G}$ and every hyperedge gives a perfect matching in $\mathcal{G}$. So $\mathcal{G}$ is a $d$-partite hypergraph that can be partitioned into $r$ perfect matchings.
\end{proof}

\section{Examples of perfect colorings of hypergraphs} \label{example-section}

\subsection{Transversals and perfect matchings}

A $k$-transversal in uniform regular hypergraphs gives us the simplest nontrivial example of perfect colorings: the vertices of a hypergraph are colored into two colors (transversal and non-transversal),  while all hyperedges have  the same color. So if $\mathcal{G}$ is a $d$-uniform $r$-regular hypergraph, then a $k$-transversal is a perfect coloring with the incidence parameters
$$V = \left( \begin{array}{c} r \\ r \end{array} \right); ~~~ W = \left( \begin{array}{cc} k & d-k \end{array} \right).$$

By Theorem~\ref{adjparam}, the parameter matrix $\mathbb{S}$ of a $k$-transversal  is the  $d$-dimensional matrix of order $2$ with entries 
$$s_{\gamma} = \left\{  \begin{array}{l}
r  {d-1 \choose k-1}^{-1} \mbox{ if } \gamma_1 = 1 \mbox{ and } \# \{ \gamma_i : \gamma_i = 1  \} = k; \\
 r  {d-1 \choose k}^{-1}\mbox{ if } \gamma_1 = 2 \mbox{ and } \# \{ \gamma_i : \gamma_i = 1  \} = k; \\
 0 \mbox{ otherwise}.
\end{array}  \right.$$

For example, the parameter matrices of $k$-transversals in $d$-uniform $r$-regular hypergraphs for some small $d$ and $k$ are the following:
$$d = 2, ~k= 1: ~~~~~~~~~~~ \mathbb{S} = \left(  \begin{array}{cc}  0 & r \\ r & 0 \end{array} \right);$$
$$d = 3, ~k= 1: ~~~~~~ \mathbb{S} = \left(  \begin{array}{cc|cc}  0 & \nicefrac{r}{2} & r & 0 \\ \nicefrac{r}{2} & 0 & 0 & 0 \end{array} \right);$$
$$d = 4, ~k= 1: ~~~~~~ \mathbb{S} = \left(  \begin{array}{cc|cc}  0 & \nicefrac{r}{3} & r & 0 \\ \nicefrac{r}{3} & 0 & 0 & 0 \\ \hline  \nicefrac{r}{3} & 0 & 0 & 0 \\ 0 & 0 & 0 & 0 \end{array} \right);$$
$$d = 4, ~k= 2: ~~~~~~ \mathbb{S} = \left(  \begin{array}{cc|cc}  0 & 0 & 0 & \nicefrac{r}{3} \\ 0 & \nicefrac{r}{3} & \nicefrac{r}{3} & 0 \\ \hline  0 & \nicefrac{r}{3} & \nicefrac{r}{3} & 0 \\ \nicefrac{r}{3} & 0 & 0 & 0 \end{array} \right).$$

Let us find eigenvalues of transversals as perfect colorings of hypergraphs.

\begin{teorema} \label{transspectrum}
Suppose that $\mathbb{S}$ is the parameter matrix of the perfect coloring corresponding to a $k$-transversal  in a $d$-uniform $r$-regular hypergraph $\mathcal{G} $. Then the eigenvalues of  $\mathbb{S}$ are $\lambda_0 = 0$ and $\lambda_j = r\xi^{jk}$, where $\xi$ is a $d$-th primitive root of unity.  In particular, the matrix $\mathbb{S}$ has  $\frac{d}{\gcd(k,d)}$ different  nonzero eigenvalues. 
\end{teorema}

\begin{proof}

Knowing the entries of the matrix $\mathbb{S}$ from Theorem~\ref{adjparam}, we find that the  equation $\mathbb{S} \circ x =  \lambda \mathbb{I} \circ x $ on eigenvalues $\lambda$ is equivalent to the following system:

$$ \left\{  \begin{array}{c}  -\lambda x_{1}^{d-1} + r x_1^{k-1} x_2^{d - k} = 0; \\
r x_1^{k} x_2^{d-k-1} - \lambda x_2^{d-1} = 0.  \end{array}  \right.
$$

From here it is easy to see that  $\lambda_0 = 0$ is an eigenvalue with the corresponding eigenvectors $(x_1, 0)$ and $(0, x_2)$. 

Assume that $x= (x_1, x_2)$ is an eigenvector with both nonzero components: $x_1 = t x_2$, $t\neq 0$. Then the above system of equations can be rewritten as
$$ \left\{  \begin{array}{l}  -\lambda t^{d-1} + r  t^{k-1} = 0; \\
r t^{k} - \lambda  = 0.  \end{array}  \right.
$$
Consequently,  $\lambda = r t^k$ and  $t^d = 1$, so $t$ is a $d$-th root of unity. So all  nonzero  eigenvalues of the matrix $\mathbb{S}$ are $\lambda = r \xi^{jk}$, $j = 1, \ldots, d$, where $\xi$ is a  $d$-th primitive root of unity. 
\end{proof}

We believe that all nonzero eigenvalues of $\mathbb{S}$ have the same algebraic multiplicity. Moreover, we propose the following conjecture  on the characteristic polynomial  of $\mathbb{S}$.

\begin{con}
Let $\mathbb{S}$ be the parameter matrix of a $k$-transversal in a $d$-uniform $r$-regular hypergraph $\mathcal{G} $. Then  the characteristic polynomial of $\mathbb{S}$ is
$$\varphi(\lambda) = \lambda^{d-2} \prod\limits_{i = j}^d ( \lambda - \xi^{jk} r),$$
where $\xi$ is a $d$-th primitive root of unity. 
\end{con}

Combining Theorem~\ref{transspectrum} with Theorem~\ref{adjcpecinclude}, we obtain that every $d$-uniform $r$-regular hypergraph with a $k$-transversal should have values $r\xi^{jk}$ in its spectrum. It gives a spectral condition on existence of $k$-transversals that cannot be reduced to the spectrum of the induced coloring of the incidence graph.

Indeed, consider complete $3$-uniform hypergraphs on $n \geq 4$ vertices. It is obvious that they do not contain transversals. 

Using computational results for~\cite{CoopDut.hyperspec} (available at~\cite{dutle.compres}), one can check that  complete $3$-uniform hypergraphs on $n$ vertices, $4 \leq n \leq 8$, do not have eigenvalues $r e^{\pm \frac{2 \pi i}{3}}$. It indicates that they do not contain  a transversal. 

On the other hand,  the parameter matrix of the induced perfect coloring  of the incidence graph of a $d$-uniform $r$-regular hypergraph is
$$\left( \begin{array}{cc} 0 & W \\ V & 0 \end{array} \right) = \left( \begin{array}{ccc} 0 & 1 & d-1 \\ r & 0 & 0 \\ r & 0 & 0  \end{array} \right).$$
The eigenvalues of this matrix are $\lambda_1 = 0$ and $\lambda_{2,3} = \pm \sqrt{dr}$, where the latter values are contained in the eigenspectrum of every  bipartite biregular graph  with degrees of parts $r$ and $d$. Eigenvalue $0$ belongs to the spectrum of  many of bipartite biregular graphs, for example, the incidence graph of the complete $3$-uniform hypergraph on $5$ vertices. 

\subsection{Perfect $2$-colorings of $3$-uniform hypergraphs}

Let us completely describe the smallest case of hypergraph perfect colorings, i.e., perfect colorings $P$ of $r$-regular $3$-uniform hypergraphs in $2$ colors.

Assume that  $n_1$ and  $n_2$ are the numbers of vertices of  first and second  colors and let $\chi = n_1 / n_2 \leq 1$. 
In  a general case, the incidence parameters $(V,W)$ of the coloring $P$ are
$$
V = \left( \begin{array}{cccc}  v_{1,1} & v_{1,2} & v_{1,3} & 0 \\  0 & v_{2,2} & v_{2,3} & v_{2,4} \end{array} \right); ~~~~ W =  \left( \begin{array}{cc}  3 & 0 \\ 2 & 1 \\ 1 & 2 \\ 0 & 3  \end{array} \right),
$$
where $v_{1,1} + v_{1,2}+ v_{1,3} = v_{2,2} + v_{2,3} + v_{2,4} = r$. If some color range of hyperedges is absent in the coloring $P$, then the matrices $W$ and $V$ lack the corresponding row  and column.

From the equality $N V = W^T M$ (see Theorem~\ref{VWblock}(1)), we deduce that
$$n_1 v_{1,2} = 2 n_2 v_{2,2}; ~~~ n_2 v_{2,3} = 2n_1 v_{1,3},$$
and, consequently,
\begin{equation} \label{2color3uni}
\frac{v_{2,2}}{v_{1,2}} = \chi /2; ~~~ \frac{v_{2,3}}{v_{1,3}} = 2 \chi.
\end{equation}

The parameter matrix of a perfect coloring $P$ is a $3$-dimensional matrix of order $2$:
$$\mathbb{S} = \left( \begin{array}{cc|cc} 
s_{1,1,1} & s_{1,1,2} & s_{2,1,1} & s_{2,1,2} \\ 
s_{1,2,1} & s_{1,2,2} & s_{2,2,1} & s_{2,2,2}  
\end{array} \right).$$

With the help of Theorem~\ref{adjparam}, we find entries of the parameter matrix $\mathbb{S}$:
\begin{gather*}
s_{1,1,1} = v_{1,1}; ~~~ s_{1,1,2} = s_{1,2,1} = v_{1,2} / 2; ~~~ s_{1,2,2} = v_{1,3};\\
s_{2,1,1} = v_{2,2}; ~~~ s_{2,1,2} = s_{2,2,1} = v_{2,3}/2; ~~~ s_{2,2,2} = v_{2,4}. 
\end{gather*}

Using additionally equalities~(\ref{2color3uni}), we reduce the number of essential parameters of the matrix $\mathbb{S}$ and denote them by $a, b ,c$ and $d$:
$$\mathbb{S} = \left( \begin{array}{cc|cc} 
a &  b & \chi b  & \chi c \\ 
 b & c  & \chi c & d 
\end{array} \right).$$

To calculate the characteristic polynomial $\varphi(\lambda)$ of the matrix $\mathbb{S}$, we use a formula from~\cite{MorShak.resultform} for resultants $R_{2,2}$:
\begin{gather*}
\varphi(\lambda) = \lambda^4  - 2\lambda^3(a + d) + \lambda^2 (d^2 + 4ad + a^2 - 6\chi bc) + \\ \lambda ((a +d) (6 \chi bc - 2ad) - 4 \chi^2 c^3 - \chi b^3) + a^2d^2 - 3 \chi^2 b^2c^2 -6\chi abcd + 4\chi^2 ac^3 + \chi b^3 d. 
\end{gather*}

\subsection{Fano's plane hypergraph}

Here we illustrate that perfect colorings can be used for finding eigenvalues of some hypergraphs and multidimensional matrices. 

Consider a  hypergraph $\mathcal{F}$ corresponding to the Fano's plane: $\mathcal{F}$ is a $3$-uniform hypergraph on $7$ vertices with the following $7$ hyperedges:
$$(x_1, x_2, x_3), ~ (x_1, x_4, x_5), ~ (x_1, x_6, x_7), ~(x_2, x_4, x_6), ~(x_2, x_5, x_7), ~(x_3, x_4, x_7), ~ (x_3, x_5, x_6).$$

Due to the symmetry of the hypergraph $\mathcal{F}$, it is not hard to list all its $2$-colorings  and  find that only two of them are perfect.

1. Color one vertex of $\mathcal{F}$ into the first color, and all other six vertices into the second color (Figure~2):
\begin{center}
\includegraphics[width=0.3\linewidth]{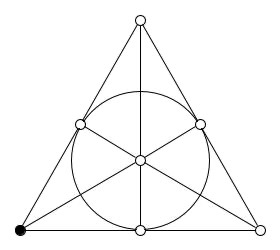}  

Figure 2: The first perfect 2-coloring of the Fano's plane hypergraph.
\end{center}

The incidence parameters are
$$V = \left( \begin{array}{cc}
3 & 0 \\ 1 & 2
\end{array} \right);  ~~~~
W = \left( \begin{array}{cc}
1 & 2 \\ 0 & 3
\end{array} \right).  $$
Using the help of Theorem~\ref{adjparam}, we find that the multidimensional parameter matrix of this perfect coloring is
$$\mathbb{S}_1 = \left( \begin{array}{cc|cc} 
0 & 0 & 0 & \nicefrac{1}{2} \\ 
0 & 3 & \nicefrac{1}{2} & 2
\end{array} \right).$$
To calculate the characteristic polynomial $\varphi_1(\lambda)$ of the matrix $\mathbb{S}_1$, we use a formula from~\cite{MorShak.resultform} or the result for $2$-colorings of $3$-uniform hypergraphs from the previous section:
$$\varphi_1(\lambda) = \lambda(\lambda-3)(\lambda^2 - \lambda +1).$$
 So the parameter matrix $\mathbb{S}_1$ has eigenvalues $0$, $3$, and $\frac{1}{2} (1 \pm i \sqrt{3})$.
 
2. Another perfect coloring of the hypergraph $\mathcal{F}$ is presented at Figure~3.
\begin{center}
\includegraphics[width=0.3\linewidth]{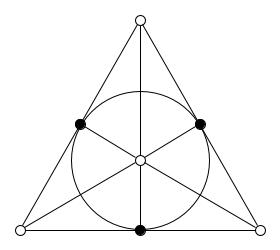}  

Figure 3: The second perfect 2-coloring of the Fano's plane hypergraph.
\end{center}

 The incidence parameters of this coloring  are
$$V = \left( \begin{array}{cc}
1 & 2 \\ 0 & 3
\end{array} \right);  ~~~~ W = \left( \begin{array}{cc}
3 & 0 \\ 1 & 2
\end{array} \right). $$
Theorem~\ref{adjparam} gives the following multidimensional parameter matrix:
$$\mathbb{S}_2 = \left( \begin{array}{cc|cc} 
1 & 0 & 0 & \nicefrac{3}{2} \\ 
0 & 2 & \nicefrac{3}{2} & 0
\end{array} \right).$$
Using the result for $2$-colorings of $3$-uniform hypergraphs from previous section, we find the characteristic polynomial  for the matrix $\mathbb{S}_2$:
$$\varphi_2(\lambda) = (\lambda-3)(\lambda - 1)(\lambda^2 + 2\lambda +6).$$
 So the parameter matrix $\mathbb{S}_2$ has eigenvalues $1$, $3$, and $-1 \pm i \sqrt{5}$.

Thus, the hypergraph $\mathcal{F}$ of the Fano's plane has at least $7$ different eigenvalues:
$$\lambda_0 = 0,~ \lambda_1 = 1, ~ \lambda_2 = 3, ~ \lambda_{3,4} = \frac{1}{2} (1 \pm i \sqrt{3}), \lambda_{5,6} = -1 \pm i \sqrt{5}.$$

\section*{Acknowledgements}

The author is grateful to the anonymous reviewer for valuable comments. 
 
This work was funded by the Russian Science Foundation under grant No 22-21-00202,  https://rscf.ru/project/22-21-00202/.


\bib{AngGard.comcover}{article}{
  author={Angluin, Dana},
  author={Gardiner, A.},
  title={Finite common coverings of pairs of regular graphs},
  journal={J. Combin. Theory Ser. B},
  volume={30},
  date={1981},
  number={2},
  pages={184--187},
  issn={0095-8956},
  doi={10.1016/0095-8956(81)90062-9},
}

\bib{our.perfcolorHam}{article}{
  author={Bespalov, Evgeny A.},
  author={Krotov, Denis S.},
  author={Matiushev, Aleksandr A.},
  author={Taranenko, Anna A.},
  author={Vorob'ev, Konstantin V.},
  title={Perfect 2-colorings of Hamming graphs},
  journal={J. Combin. Des.},
  volume={29},
  date={2021},
  number={6},
  pages={367--396},
  issn={1063-8539},
  doi={10.1002/jcd.21771},
}

\bib{BorRifaZin.compregcodes}{article}{
  author={Borzhes, Zh.},
  author={Rifa, Zh.},
  author={Zinov\cprime ev, V. A.},
  title={On completely regular codes},
  language={Russian, with Russian summary},
  journal={Problemy Peredachi Informatsii},
  volume={55},
  date={2019},
  number={1},
  pages={3--50},
  translation={ journal={Probl. Inf. Transm.}, volume={55}, date={2019}, number={1}, pages={1--45}, },
  doi={10.1134/S0032946019010010},
}

\bib{ChaQiZhang.survspecthyper}{article}{
  author={Chang, Kungching},
  author={Qi, Liqun},
  author={Zhang, Tan},
  title={A survey on the spectral theory of nonnegative tensors},
  journal={Numer. Linear Algebra Appl.},
  volume={20},
  date={2013},
  number={6},
  pages={891--912},
  issn={1070-5325},
  doi={10.1002/nla.1902},
}

\bib{coop.specrandcomp}{article}{
  author={Cooper, Joshua},
  title={Adjacency spectra of random and complete hypergraphs},
  journal={Linear Algebra Appl.},
  volume={596},
  date={2020},
  pages={184--202},
  issn={0024-3795},
  doi={10.1016/j.laa.2020.03.013},
}

\bib{CoopDut.hyperspec}{article}{
  author={Cooper, Joshua},
  author={Dutle, Aaron},
  title={Spectra of uniform hypergraphs},
  journal={Linear Algebra Appl.},
  volume={436},
  date={2012},
  number={9},
  pages={3268--3292},
  issn={0024-3795},
  doi={10.1016/j.laa.2011.11.018},
}

\bib{CveDoobSac.gspectra}{book}{
  author={Cvetkovi\'{c}, Drago\v {s} M.},
  author={Doob, Michael},
  author={Sachs, Horst},
  title={Spectra of graphs},
  edition={3},
  note={Theory and applications},
  publisher={Johann Ambrosius Barth, Heidelberg},
  date={1995},
  pages={ii+447},
  isbn={3-335-00407-8},
}

\bib{dels.assch}{article}{
  author={Delsarte, P.},
  title={An algebraic approach to the association schemes of coding theory},
  journal={Philips Res. Rep. Suppl.},
  date={1973},
  number={10},
  pages={vi+97},
}

\bib{dutle.compres}{article}{
  author={Dutle, Aaron},
  title={Computational results for spectra of hypergraphs},
  eprint={https://people.math.sc.edu/cooper/spectraofhypergraphs.pdf},
}

\bib{GeKaZe.multdet}{book}{
  author={Gelfand, I. M.},
  author={Kapranov, M. M.},
  author={Zelevinsky, A. V.},
  title={Discriminants, resultants and multidimensional determinants},
  series={Modern Birkh\"{a}user Classics},
  note={Reprint of the 1994 edition},
  publisher={Birkh\"{a}user Boston, Inc., Boston, MA},
  date={2008},
  pages={x+523},
  isbn={978-0-8176-4770-4},
}

\bib{godsil.algcomb}{book}{
  author={Godsil, C. D.},
  title={Algebraic combinatorics},
  series={Chapman and Hall Mathematics Series},
  publisher={Chapman \& Hall, New York},
  date={1993},
  pages={xvi+362},
  isbn={0-412-04131-6},
}

\bib{GodRoy.alggrath}{book}{
  author={Godsil, Chris},
  author={Royle, Gordon},
  title={Algebraic graph theory},
  series={Graduate Texts in Mathematics},
  volume={207},
  publisher={Springer-Verlag, New York},
  date={2001},
  pages={xx+439},
  isbn={0-387-95241-1},
  isbn={0-387-95220-9},
  doi={10.1007/978-1-4613-0163-9},
}

\bib{GraGroLo.handbook}{collection}{
  title={Handbook of combinatorics. Vol. 1, 2},
  editor={Graham, R. L.},
  editor={Gr\"{o}tschel, M.},
  editor={Lov\'{a}sz, L.},
  publisher={Elsevier Science B.V., Amsterdam; MIT Press, Cambridge, MA},
  date={1995},
  pages={Vol. I: cii+1018 pp.; Vol. II: pp. i--cii and 1019--2198},
  isbn={0-444-88002-X},
}

\bib{HuYe.multeigen}{article}{
  author={Hu, Shenglong},
  author={Ye, Ke},
  title={Multiplicities of tensor eigenvalues},
  journal={Commun. Math. Sci.},
  volume={14},
  date={2016},
  number={4},
  pages={1049--1071},
  issn={1539-6746},
  doi={10.4310/CMS.2016.v14.n4.a9},
}

\bib{leighton.comcover}{article}{
  author={Leighton, Frank Thomson},
  title={Finite common coverings of graphs},
  journal={J. Combin. Theory Ser. B},
  volume={33},
  date={1982},
  number={3},
  pages={231--238},
  issn={0095-8956},
  doi={10.1016/0095-8956(82)90042-9},
}

\bib{LiHou.hypercover}{article}{
  author={Li, Deqiong},
  author={Hou, Yaoping},
  title={Hypergraph coverings and their zeta functions},
  journal={Electron. J. Combin.},
  volume={25},
  date={2018},
  number={4},
  pages={Paper No. 4.59, 15},
}

\bib{lim.eigentensor}{article}{
  author={Lim, L.-H.},
  title={Singular values and eigenvalues of tensors: A variational approach},
  conference={ title={1st IEEE International Workshop}, },
  book={ series={Computational Advances in Multi-Sensor Adapative Processing}, publisher={IEEE, Piscataway, NJ}, },
  date={2005},
  pages={129--132},
}

\bib{MorShak.resultform}{article}{
  author={Morozov, A.},
  author={Shakirov, Sh.},
  title={Analogue of the identity Log Det = Trace Log for resultants},
  journal={J. Geom. Phys.},
  volume={61},
  date={2011},
  number={3},
  pages={708--726},
  issn={0393-0440},
  doi={10.1016/j.geomphys.2010.12.001},
}

\bib{PotAv.hypercolor}{article}{
  author={Potapov, V. N.},
  author={Avgustinovich, S. V.},
  title={Combinatorial designs, difference sets and bent functions as perfect colorings of graphs and multigraphs},
  language={Russian, with Russian summary},
  journal={Sibirsk. Mat. Zh.},
  volume={61},
  date={2020},
  number={5},
  pages={1087--1100},
  issn={0037-4474},
  translation={ journal={Sib. Math. J.}, volume={61}, date={2020}, number={5}, pages={867--877}, issn={0037-4466}, },
  doi={10.33048/smzh.2020.61.510},
}

\bib{qi.eigentensor}{article}{
  author={Qi, Liqun},
  title={Eigenvalues of a real supersymmetric tensor},
  journal={J. Symbolic Comput.},
  volume={40},
  date={2005},
  number={6},
  pages={1302--1324},
  issn={0747-7171},
  doi={10.1016/j.jsc.2005.05.007},
}

\bib{shao.tensprod}{article}{
  author={Shao, Jia-Yu},
  title={A general product of tensors with applications},
  journal={Linear Algebra Appl.},
  volume={439},
  date={2013},
  number={8},
  pages={2350--2366},
  issn={0024-3795},
  doi={10.1016/j.laa.2013.07.010},
}

\bib{ShaoShanZhang.detoftens}{article}{
  author={Shao, Jia-Yu},
  author={Shan, Hai-Ying},
  author={Zhang, Li},
  title={On some properties of the determinants of tensors},
  journal={Linear Algebra Appl.},
  volume={439},
  date={2013},
  number={10},
  pages={3057--3069},
  issn={0024-3795},
  doi={10.1016/j.laa.2013.08.014},
}

\bib{SongFanWang.hyperspeccover}{article}{
  author={Song, Yi-Min},
  author={Fan, Yi-Zheng},
  author={Wang, Yi},
  author={Tian, Meng-Yu},
  author={Wan, Jiang-Chao},
  title={The spectral property of hypergraph coverings},
  journal={Discrete Math.},
  volume={347},
  date={2024},
  number={3},
  pages={Paper No. 113830, 15},
  issn={0012-365X},
  doi={10.1016/j.disc.2023.113830},
}

\bib{my.perfstrct}{article}{
  author={Taranenko, Anna A.},
  title={Algebraic properties of perfect structures},
  journal={Linear Algebra Appl.},
  volume={607},
  date={2020},
  pages={286--306},
  issn={0024-3795},
  doi={10.1016/j.laa.2020.08.012},
}

\bib{WeiLeh.algo}{article}{
  author={Weisfeiler, B.},
  author={Leman, A.},
  title={The reduction of a graph to canonical form and the algebra which appears therein},
  journal={Nauchno-Technicheskaya Informatsia},
  volume={2},
  number={9},
  date={1968},
  pages={12--16},
  language={Russian},
}



\begin{bibdiv}
    \begin{biblist}[\normalsize]
    \bibselect{biblio}
    \end{biblist}
    \end{bibdiv}
\end{document}